\documentclass[12pt]{amsart}%

\usepackage{amsfonts}
\usepackage{amsmath}
\usepackage{amssymb}
\usepackage{amsthm}
\usepackage{enumitem}
\usepackage{mathrsfs}
\usepackage{graphicx}%
\usepackage{amsmath}
\usepackage{setspace}
\usepackage{hyperref}
\usepackage{cite}
\usepackage{amsfonts}
\usepackage{algorithmic}
\usepackage{textcomp}
\allowdisplaybreaks[4]
\makeatletter
\@addtoreset{equation}{section}
\makeatother
\renewcommand{\theequation}{\arabic{section}.\arabic{equation}}
\newtheorem{theorem}{Theorem}[section]

\newtheorem{corollary}[theorem]{Corollary}

\newtheorem{lemma}[theorem]{Lemma}

\newtheorem{proposition}[theorem]{Proposition}
\newtheorem{remark}[theorem]{Remark}

\newtheorem*{theoremA}{Theorem A}
\newtheorem*{theoremB}{Theorem B}
\newtheorem*{theoremC}{Theorem C}

\newcommand{\Sn}{\mathbb S^n}
\newcommand{\Rn}{\mathbb R^n}
\newcommand{\R}{\mathbb R}

\newcommand{\cM}{\mathcal M}
\newcommand{\cX}{\mathcal X}
\newcommand{\cD}{\mathcal D}
\newcommand{\cQ}{\mathcal Q}
\newcommand{\cH}{\mathcal H}

\newcommand{\dist}{\operatorname{dist}}
\newcommand{\supp}{\operatorname{supp}}
\newcommand{\weak}{\rightharpoonup}
\newcommand{\eps}{\varepsilon}
\newcommand{\dd}{\,\mathrm d}
\newcommand{\one}{\mathbf 1}
\newcommand{\abs}[1]{\lvert#1\rvert}
\newcommand{\norm}[1]{\lVert#1\rVert}
\newcommand{\ip}[2]{\langle #1,#2\rangle}
\makeatletter
\def\@makefnmark{}
\makeatother

\begin{document}

\title[Attainment of the sharp log-Sobolev inequality]{Attainment of the sharp stability constant of Beckner conformal log-Sobolev inequality}

\author{Lu Chen}
\address[Lu Chen] {Key Laboratory of Algebraic Lie Theory and Analysis of Ministry of Education,
School of Mathematics and Statistics, Beijing Institute of Technology,
Beijing 100081, China}
\email{chenlu5818804@163.com}

\author{Bohan Wang}
\address[Bohan Wang]{Key Laboratory of Algebraic Lie Theory and Analysis of Ministry of Education, School of Mathematics and Statistics, Beijing Institute of Technology, Beijing
100081, PR China}
\email{wangbohanbit2022@163.com}


\thanks{
The first author was partly supported by National Natural Science Foundation of China (No. 12622108), a grant from Bejing Institute of Technology (No.2022CX01002). The second author was supported by the BIT Research and Innovation Promoting Project(Grant No.2025YCXY030) }

\begin{abstract}
The existence of minimizer for stability of the
fractional Sobolev inequality and the stability of Beckner conformal log-Sobolev inequality (the endpoint case of fractional Sobolev inequality) have been established by K\"onig and Chen-Lu-Tang.  The existence of minimizer for stability of the log-Sobolev inequality is unsolved. In this paper, we solve this problem. For every $n\geq2$, we prove that the sharp $L^2$-stability
constant  of  Beckner conformal log-Sobolev inequality
on $\mathbb S^n$ is attained and is strictly less than $\frac{2\log2}{n}$.
Moreover, the normalized stability constant of the
fractional Sobolev inequality
converges to the stability constant of the
log-Sobolev inequality  as $s\to0^+$.
After Hersch normalization, every sequence of
$L^{p_{s_k}}$-normalized minimizers for stability of the fractional Sobolev inequality, with $s_k\to0^+$, has a subsequence converging
strongly in $L^2(\mathbb S^n)$ to an $L^2$-normalized minimizer of stability of the  log-Sobolev inequality.
\end{abstract}

\maketitle

\section{Introduction}
For $n\geq3$, the stability problem for the sharp Sobolev inequality asks
whether its deficit controls the squared $\dot H^1(\Rn)$-distance to the
manifold of optimizers. Brezis and Lieb first \cite{BrezisLieb1985} proposed this
problem in the form
\[
 \|\nabla u\|_{L^2(\Rn)}^2
 -\mathcal S_n\|u\|_{L^{\frac{2n}{n-2}}(\Rn)}^2
 \geq c_0\,\dist_{\dot H^1(\Rn)}^2(u,\cM_{\mathrm{Sob}}),
\]
where $\mathcal S_n$ is the sharp Sobolev constant,
\[
 \dist_{\dot H^1(\Rn)}^2(u,\cM_{\mathrm{Sob}})
 :=\inf_{h\in\cM_{\mathrm{Sob}}}
 \|\nabla(u-h)\|_{L^2(\Rn)}^2,
\]
and
\[
 \cM_{\mathrm{Sob}}
 :=\left\{c(a+|x-x_0|^2)^{-\frac{n-2}{2}}:
 a>0,\ c\in\R\setminus\{0\},\ x_0\in\Rn\right\}.
\]
Bianchi and Egnell \cite{BianchiEgnell1991} proved this estimate with
$c_0>0$. Higher-order and fractional extensions were obtained in
\cite{LuWei1999, BartschWethWillem2003, ChenFrankWeth2013}. Explicit and
dimension-dependent bounds for stability constants of the above geometric inequality
were furthermore established in \cite{DolbeaultEtAl2025,ChenLuTang2024,ChenLuTang2025,ChenLuTang2026MA,ChenLuTang2026CP}.

Lieb  established  Hardy-Littlewood-Sobolev inequality   in \cite{Lieb},  the stability of the HLS inequality has recieved much less attention. Later, Carlen \cite{Carlen2017} established the stability of the HLS inequality by a dual stability method. Due to the non-Hilbert distance $\|\cdot\|_{L^{\frac{2n}{n+2s}}}$, it seems difficult to prove the stability of the HLS inequality directly. Recently, the authors of \cite{ChenLuTang2026MA} developed a method based on the $H^{-s}$-decomposition instead of $L^{\frac{2n}{n+2s}}$-decomposition to obtain the local stability of the HLS inequality, then directly established the stability of the HLS inequality with the optimal order and dimension-dependent constants.
\medskip

Log-Sobolev inequalities originate in the work of Gross
\cite{Gross1975} and play a central role in the entropy and
hypercontractivity theory of Markov semigroups, see
\cite{BakryGentilLedoux2014}. 
We consider Beckner conformally invariant inequality on the sphere. This inequality follows by differentiating the sharp fractional Sobolev inequality with respect to the order parameter at $s=0$. This connection suggests studying the stability problem for the log-Sobolev inequality through the positive-order theory.
 The key is to obtain estimates that remain
 uniform  as  $s\to0^+$.
\medskip

Throughout the paper, all functions are real-valued.
Let $\dd\sigma=|\mathbb S^n|^{-1}d\omega$ be the
probability measure on $\Sn$, and let $\{Y_{\ell,m}\}$ be an orthonormal
basis of spherical harmonics. We write
$\|\cdot\|_p=\|\cdot\|_{L^p(\Sn)}$ and set
\[
 p_s:=\frac{2n}{n-2s},\qquad 0\leq s<\frac{n}{2},
 \qquad
 \mathbb B^{n+1}:=\{\xi\in\R^{n+1}:|\xi|<1\}.
\]
For $0<s<\frac{n}{2}$, Beckner \cite{Beckner1992,Beckner1993}
proved that
\begin{equation}\label{Beckner1}
 \|A_{2s}^{\frac{1}{2}}u\|_2^2
 \geq\mathcal S_{s,n}\|u\|_{p_s}^2, \qquad  u\in H^s(\Sn),
\end{equation}
where
\[
 A_{2s}Y_{\ell,m}
 =\frac{\Gamma(\ell+\frac{n}{2}+s)}
 {\Gamma(\ell+\frac{n}{2}-s)}Y_{\ell,m},
 \qquad
 \mathcal S_{s,n}=\frac{\Gamma(\frac{n}{2}+s)}
 {\Gamma(\frac{n}{2}-s)}.
\]
For $0\leq s<\frac{n}{2}$, define
\begin{equation}\label{Beckner-extremal}
 h_{\xi,s}(\omega)
 :=\left(\frac{\sqrt{1-|\xi|^2}}{1-\xi\cdot\omega}\right)^{\frac{n-2s}{2}},
 \qquad \xi\in\mathbb B^{n+1},
\end{equation}
and the optimizer manifold for the fractional Sobolev inequality
\[
 \cM_s:=\{c h_{\xi,s}:c\in\R^1,\ \xi\in\mathbb B^{n+1}\}.
\]
For $s>0$, equality in \eqref{Beckner1} holds if and only if
$u\in\cM_s$.
At $s=0$, both sides of    \eqref{Beckner1} coincide.
Differentiation  gives the logarithmic form with spectral multipliers
\[
 \cH Y_{\ell,m}=\Lambda_\ell Y_{\ell,m},
 \qquad
 \Lambda_\ell=
 \psi\!\left(\ell+\frac{n}{2}\right)
 -\psi\!\left(\frac{n}{2}\right),
\]
where $\psi=\frac{\Gamma'}{\Gamma}$ is the digamma function. Its form domain
\[
 \cX:=\{u\in L^2(\Sn):\ip{u}{\cH u}<\infty\},
\]
where $\ip{u}{\cH u}$ is understood in the quadratic-form sense, and
define
\[
 \cM
 :=
 \left\{
 cJ_\Phi^{\frac12}:
 c\in\mathbb R^1,\
 \Phi\in\operatorname{Conf}(\Sn)
 \right\},
\]
 where $J_\Phi$ is the Jacobian of $\Phi$ with respect to the standard volume measure and
$\operatorname{Conf}(\Sn)$ denotes the M\"obius conformal
automorphism group of the  sphere.

For $u\in\cX$, define
\begin{equation}\label{eq:D-intro}
 \cD(u):=\ip{u}{\cH u}
 -\frac{1}{n}\int_{\Sn}u^2
 \log\frac{u^2}{\norm{u}_2^2}\,\dd\sigma,
\end{equation}
with $\cD(0)=0$. The $L^2$-unitary conformal action preserves
$\cD$ and $\cM$. Beckner inequality
\cite{Beckner1992,Beckner1993,Beckner1997} reads as follows.

\begin{theoremA}
For every $n\geq1$ and every $u\in\cX$,
\begin{equation*}
 \cD(u)\geq0.
\end{equation*}
Equality holds if and only if $u\in\cM$.
\end{theoremA}
We also recall that Frank-K\"onig-Tang \cite{FrankKonigTang2020}
showed that all positive finite-energy
solutions to the Euler-Lagrange equation
associated with the log-Sobolev inequality on the sphere belong to
 $c h_{\xi,0}$, $c>0$.
The sharp stability constant of  the log-Sobolev inequality
is
\begin{equation*}
 c_{\rm LS}:=
 \inf_{u\in\cX\setminus\cM}
 \frac{\cD(u)}{\dist(u,\cM)^2},  \qquad \dist(u,\cM):=\inf\limits_{h\in\cM}\norm{u-h}_{2}.
\end{equation*}
Chen-Lu-Tang \cite{ChenLuTang2023}
proved that the sharp
constant of global stability of   log-Sobolev inequality
 is strictly less than its sharp constant
of local stability.
They subsequently established global stability for the log-Sobolev inequality on the sphere by applying the endpoint differentiation argument to the  stability estimates for the fractional Sobolev inequality in
\cite{ChenLuTang2025}. Their results can be stated as follows.

\begin{theoremB}
For every $n\geq1$ there exists $\alpha_n>0$ such that
\[
 \cD(u)\geq\alpha_n\dist(u,\cM)^2,
 \qquad u\in\cX.
\]
In particular, $c_{\rm LS}>0$. If $n\geq2$, then
\begin{equation}\label{eq:strict-local-intro}
 c_{\rm LS}<\frac{2}{n+2},
\end{equation}
where $\frac{2}{n+2}$ is the sharp local stability constant   transverse to
$\cM$.
\end{theoremB}

These results  do not determine
whether the
infimum is attained. For $n\geq2$ and each fixed
positive order, the corresponding attainment problem was settled by
K\"onig \cite{Konig2023,Konig2025}. To state his result in our normalization,
set
\begin{equation}\label{eq:intro-fractional-normalization}
 P_s=\frac{\Gamma(\frac{n}{2}-s)}{\Gamma(\frac{n}{2}+s)}A_{2s},
 \qquad
 Q_s(u)=\ip{u}{P_su},
\end{equation}
where the pairing is understood in the quadratic-form sense and
the form domain is $\mathcal X_s:=H^s(\mathbb S^n)=\{u:\langle u,P_s u\rangle<\infty\}$.
Thus $P_s\one=\one$ and $Q_s(u)\geq\norm{u}_{p_s}^2$.
 With
$\cM_s$ as defined above, set
\[
 \dist_s(u,\cM_s):=\inf_{h\in\cM_s}Q_s(u-h)^{\frac{1}{2}},
\]
and define
\[
 \cD_s(u):=Q_s(u)-\norm{u}_{p_s}^2,
 \qquad
 c_s:=\frac{1}{2s}\inf_{u\in\mathcal X_s\setminus\cM_s}
 \frac{\cD_s(u)}{\dist_s(u,\cM_s)^2}.
\]
Thus $c_s$ is the normalized stability constant of the fractional Sobolev inequality of order $s$.
Minimizers are normalized by $\norm{u}_{p_s}=1$.
\medskip

\begin{theoremC}
Let $n\geq2$ and $0<s<\frac{n}{2}$. The constant $c_s$ is attained, and
normalized minimizing sequences are precompact modulo conformal
transformations. Moreover,
\begin{equation}\label{eq:Konig-thresholds-intro}
 c_s<\frac{2}{n+2+2s},
 \qquad
 c_s<\frac{2-2^{\frac{n-2s}{n}}}{2s}.
\end{equation}
\end{theoremC}

The two bounds in \eqref{eq:Konig-thresholds-intro} correspond to the
one-bubble and two-bubble thresholds, whose limits are
$\frac{2}{n+2}$ and $\frac{2\log2}{n}$, respectively.
More generally, blow-up and compactness phenomena for fixed-order
conformally invariant nonlocal equations have been studied
in the fractional Nirenberg problem \cite{JinLiXiong2014}.  Such fixed-order results do not by themselves settle the present limit.  As
$s\to0^+$, a bound on $\frac{\cD_s(u)}{2s}$
 does not   separately control the logarithmic energy and the entropy, so uniform
compactness of the fractional minimizers requires additional estimates.
\medskip

We prove that $c_{\rm LS}$ is attained for every $n\geq2$.
First, we prove a strict upper bound below the two-bubble threshold.
We obtain this bound from the leading interaction of two conformal
extremals. We use the Euclidean realizations $\cD_{\Rn}$ and $\cM_{\Rn}$ of
$\cD$ and $\cM$,   defined  in
Section \ref{sec:two-bubble}.

\begin{theorem}\label{thm:two-bubble-intro}
Let $n\geq1$. In the Euclidean case 
of Theorem A, let
\[
 B(x)=b_n(1+|x|^2)^{-\frac{n}{2}},\qquad
 b_n^2=\frac{\Gamma(n)}{\pi^{\frac{n}{2}}\Gamma(\frac{n}{2})},\qquad
 B_\lambda(x)=\lambda^{\frac{n}{2}}B(\lambda x).
\]
Then, as $\lambda\to0^+$,
\begin{equation}\label{eq:two-bubble-main-intro}
 \begin{split}
 \frac{\cD_{\Rn}(B+B_\lambda)}{\dist(B+B_\lambda,\cM_{\Rn})^2}
 & =\frac{2\log2}{n}
 -\frac{\Gamma(n)}{\Gamma(\frac{n}{2})^2}
 \lambda^{\frac{n}{2}}\left(\log\frac{1}{\lambda}\right)^2+O_n\!\left(\lambda^{\frac{n}{2}}\log\frac{1}{\lambda}\right).
 \end{split}
\end{equation}
In particular,
\begin{equation}\label{eq:cLS-c2b-intro}
 c_{\rm LS}<\frac{2\log2}{n}.
\end{equation}
\end{theorem}

Together with  $\limsup\limits_{s\to0^+}c_s\leq c_{\rm LS}$
 proved in Section \ref{sec:prelim}, the strict
two-bubble bound \eqref{eq:cLS-c2b-intro} yields the uniform gap
 below the fractional two-bubble
threshold for all sufficiently small $s$. Then we prove uniform compactness and obtain our main result.

\begin{theorem}\label{thm:main}
Let $n\geq2$. There exists a minimizer $u\in\cX\setminus\cM$ such that
\begin{equation}\label{eq:main-attainment}
 c_{\rm LS}=\frac{\cD(u)}{\dist(u,\cM)^2}.
\end{equation}
Moreover,
\[
 \lim_{s\to0^+}c_s=c_{\rm LS}.
\]
Finally, if $s_k\to0^+$, and for each $k$, $u_k$ is a normalized
minimizer for $c_{s_k}$, then after Hersch normalization, there exist
a subsequence $\{k_j\}$ and a normalized minimizer $u_\infty$ for
$c_{\rm LS}$ such that
\[
 u_{k_j}\longrightarrow u_\infty
 \quad\text{strongly in }L^2(\Sn).
\]
\end{theorem}

The proof proceeds through two different compactness steps. First,
slowly varying cutoffs in a conformal cylinder localize the energy
with $o(s)$ errors, including across relative conformal scales of
order $e^{\frac{c}{s}}$.  Concentration-compactness gives two alternatives:
tightness or splitting. The uniform two-bubble gap rules out splitting.  Hersch normalization prevents the entire critical mass from drifting to a
pole, which   yields uniform non-concentration on the sphere. We  establish weak nonvanishing and rule out
 sign cancellation  in Section \ref{sec:uniform-boundedness}.

Non-concentration alone does not exclude sign cancellation or
unbounded maxima. We  derive an Euler-Lagrange equation using
the whole active-set of nearest bubbles, since the nearest bubble
need not be unique. Adaptive localization and a sign-mixing estimate
produce a nonzero weak limit and keep all active bubble parameters in a compact set. 
If the maximum were unbounded, the inverse-kernel estimate with
near-field coefficient $1+O(s)$ would yield a fixed amount of mass
in a shrinking ball. This would contradict non-concentration.
 The resulting
uniform $L^\infty$ bound separates the energy and entropy and gives
strong $L^2$ compactness.
The one-bubble lower bound, together with the strict local gap,
excludes convergence to the optimizer manifold. The recovery inequality and
lower-limit inequality yield attainment and
$c_s\to c_{\rm LS}$.
\medskip

The paper is organized as follows. Section \ref{sec:prelim} establishes
the preliminaries. Section \ref{sec:two-bubble}
proves Theorem \ref{thm:two-bubble-intro}.
Sections \ref{sec:localization} and \ref{sec:uniform-boundedness} establish
uniform non-concentration and boundedness of minimizers for $c_s$.
Section \ref{sec:attainment-proof} completes the proof of
Theorem \ref{thm:main}. The appendices contain the auxiliary proofs.
\medskip

\section{Preliminaries}\label{sec:prelim}
As $s\to0^+$, the fractional stability problem approaches the conformal log-Sobolev problem. 
We   establish recovery for the deficit and the
distance to the optimizer manifold.
  For every smooth test function, we first
  show that the  stability
quotient of the fractional Sobolev inequality converges to the  stability quotient of log-Sobolev inequality,
which yields the upper bound $\limsup\limits_{s\to0^+}c_s\leq c_{\rm LS}$.

Let $\mathcal H_\ell$ be the space of spherical harmonics of degree $\ell$,
and choose an orthonormal basis $\{Y_{\ell,m}\}_{m=1}^{N(n,\ell)}$.
Every  $u\in L^2(\Sn)$ has the expansion
\[u=\sum\limits_{\ell=0}^\infty\sum\limits_{m=1}^{N(n,\ell)}u_{\ell,m}Y_{\ell,m},\]
and the logarithmic form is
\[ \ip{u}{\cH u}=\sum\limits_{\ell=0}^\infty\Lambda_\ell
 \sum\limits_{m=1}^{N(n,\ell)}\abs{u_{\ell,m}}^2.
\] We equip its domain $\cX$ with the form norm
\begin{equation*}
 \norm{u}_{\cX}^2=\norm{u}_2^2+\ip{u}{\cH u}.
\end{equation*}
Since $\Lambda_\ell\to\infty$, the embedding
$\cX\hookrightarrow L^2(\Sn)$ is compact.
We will use this compactness after establishing a uniform
form-norm bound for the fractional minimizers.
\medskip

Order differentiation controls the numerator of the stability quotient. The
denominator is more delicate because it involves minimization over the full
conformal optimizer manifold. We therefore introduce the $L^2$-normalized
optimizer family
\begin{equation}\label{eq:endpoint-bubble-param}
 \cM_{0,1}:=\{h\in\cM:\norm{h}_2=1\}
 =\{\pm h_{\xi,0}:\xi\in\mathbb B^{n+1}\},
\end{equation}
where $h_{\xi,0}$ is given by \eqref{Beckner-extremal} at $s=0$.
Minimizing along each line in $\cM$ gives the following distance formula.

\begin{lemma}\label{lem:distance-formula}
For every $u\in L^2(\Sn)$,
\begin{equation}\label{eq:distance-formula}
 \dist(u,\cM)^2=\norm{u}_2^2-m(u),
 \qquad
 m(u):=\sup_{h\in\cM_{0,1}}\abs{\ip{u}{h}}^2.
\end{equation}
If $m(u)>0$, the supremum is attained.
\end{lemma}

\begin{proof}
For each $h\in\cM_{0,1}$,
\[
 \inf_{c\in\mathbb R^1}\|u-ch\|_2^2
 =\|u\|_2^2-|\langle u,h\rangle|^2,
\]
and taking the infimum over $h$ gives \eqref{eq:distance-formula}. If
$m(u)>0$, choose $\xi_k\in\mathbb B^{n+1}$ with
$|\langle u,h_{\xi_k,0}\rangle|^2\to m(u)$ and pass to
$\xi_k\to\xi\in\overline{\mathbb B^{n+1}}$. The case $|\xi|=1$ is
impossible, $h_{\xi_k,0}\to0$ uniformly away from $\xi$, and absolute
continuity of the $L^2$ integral gives
$\langle u,h_{\xi_k,0}\rangle\to0$. Hence $|\xi|<1$, and the explicit
formula gives $h_{\xi_k,0}\to h_{\xi,0}$ strongly in $L^2$, so the
supremum is attained.
\end{proof}

The endpoint differentiation is immediate for smooth functions, while the sharp
constant problem is posed on the form domain $\cX$. We therefore require a
smooth approximation that preserves both the quadratic form and the normalized
entropy. At $u=0$, the normalized entropy and $\cD(0)$ are understood by
continuity to be zero.

\begin{lemma}\label{lem:smooth-recovery-density}
For every $u\in\cX$, there is a sequence
$v_j\in C^\infty(\Sn)$ such that
\begin{align*}
 \norm{v_j-u}_{\cX}&\longrightarrow0,
  \qquad
 \int_{\Sn}v_j^2\log\frac{v_j^2}{\norm{v_j}_2^2}\,\dd\sigma
 \longrightarrow
 \int_{\Sn}u^2\log\frac{u^2}{\norm{u}_2^2}\,\dd\sigma.
\end{align*}
Consequently,
\begin{equation*}
 \cD(v_j)\longrightarrow\cD(u),
 \qquad
 \dist(v_j,\cM)\longrightarrow\dist(u,\cM).
\end{equation*}
If $u\notin\cM$, the sequence may be chosen in
$C^\infty(\Sn)\setminus\cM$ and
\begin{equation}\label{eq:smooth-quotient-density}
 c_{\rm LS}
 =\inf_{v\in C^\infty(\Sn)\setminus\cM}
 \frac{\cD(v)}{\dist(v,\cM)^2}.
\end{equation}
\end{lemma}

The proof in Appendix A uses heat-semigroup approximation for the form and
Jensen inequality for the positive part of the entropy. We now compare this
endpoint structure with the normalized fractional problem. For
$0<s<\frac{n}{2}$, recall the operator $P_s$ introduced in
\eqref{eq:intro-fractional-normalization}.
Its eigenvalue on $\mathcal H_\ell$ is
\begin{equation}\label{eq:frac-eigen}
 r_{s,\ell}
 :=
 \frac{\Gamma(\ell+\frac n2+s)}
      {\Gamma(\ell+\frac n2-s)}
 \frac{\Gamma(\frac n2-s)}
      {\Gamma(\frac n2+s)},
 \qquad
 P_sY_{\ell,m}=r_{s,\ell}Y_{\ell,m}.
\end{equation}
\medskip

For $s>0$, the corresponding $L^{p_s}$-normalized family is
\begin{equation}\label{eq:fractional-bubble-param}
 \cM_{s,1}:=\{h\in\cM_s:\norm{h}_{p_s}=1\}
 =\{\pm h_{\xi,s}:\xi\in\mathbb B^{n+1}\}.
\end{equation}
The second index denotes unit normalization.
Since
$\frac{n-2s}{2}p_s=n$, the density $h_{\xi,s}^{p_s}$
is independent of $s$ and is the  Jacobian of a M\"obius
transformation. Hence the normalization holds also at $s=0$.
For a signed normalized optimizer $h\in\cM_{s,1}$, define
\begin{equation}\label{eq:signed-dual-map}
 \mathfrak g_s(h):=|h|^{p_s-2}h.
\end{equation}
The normalization in
\eqref{eq:intro-fractional-normalization} and the Euler-Lagrange
equation for the fractional extremals imply that
\[
 P_sh_{\xi,s}=h_{\xi,s}^{p_s-1},
 \qquad
 \norm{h_{\xi,s}}_{p_s}=Q_s(h_{\xi,s})=1.
\]

Its dual weight is therefore
\begin{equation}\label{eq:fractional-dual-weight}
 g_{\xi,s}:=h_{\xi,s}^{p_s-1}
 =\left(\frac{\sqrt{1-|\xi|^2}}{1-\xi\cdot\omega}\right)^{\frac{n+2s}{2}},
 \qquad \norm{g_{\xi,s}}_{p_s'}=1.
\end{equation}
In particular, $\mathfrak g_s(\pm h_{\xi,s})=\pm g_{\xi,s}$. Define
\[
 m_s(u):=\sup_{h\in\cM_{s,1}}
 \left|\int_{\Sn}u\mathfrak g_s(h)\,\dd\sigma\right|^2=\sup_{\xi\in \mathbb B^{n+1}}
 \left|\int_{\Sn}u g_{\xi,s}\,\dd\sigma\right|^2.
\]
As in Lemma \ref{lem:distance-formula}, minimizing first in the scalar
multiple of a normalized optimizer yields
\begin{equation}\label{eq:ms-ds}
 \dist_s(u,\cM_s)^2=Q_s(u)-m_s(u).
\end{equation}
For fixed $s>0$ and $n\geq2$, the infimum defining $c_s$ is attained
\cite{Konig2025}.  The order-differentiation argument of
\cite{ChenLuTang2025} gives, for every fixed smooth $u$, as $s\to0^+$,
\begin{equation}\label{eq:fixed-numerator-recovery}
 \frac{1}{2s}\cD_s (u)\longrightarrow\cD(u).
\end{equation}
The denominator is more delicate. It involves maximal correlations over
the full optimizer family. We therefore establish their convergence.

\begin{lemma}
\label{lem:fixed-distance-recovery}
For every $u\in C^\infty(\Sn)$, as $ s\rightarrow0 ^+$,
\begin{equation}\label{eq:fixed-distance-recovery}
 \dist_s(u,\cM_s)^2\longrightarrow\dist(u,\cM)^2.
\end{equation}
\end{lemma}
\begin{proof}
 The explicit formulas give
$g_{\xi,s}=h_{\xi,0}^{1+\frac{2s}{n}}$ and $\|h_{\xi,0}\|_2=1$. For
$0\le s\le\frac n4$,
\[
 |x^{1+\frac{2s}{n}}-x|\le  \frac{2s}{n}C(1+x^2),\qquad x\ge0,
\]
hence, uniformly in $\xi$,
\[
 \|g_{\xi,s}-h_{\xi,0}\|_1
 \le C_n s.
\]
Therefore, for fixed smooth $u$,
\[
 \bigl|m_s(u)^{\frac{1}{2}}-m(u)^{\frac{1}{2}}\bigr|
 \le C_ns\|u\|_\infty\rightarrow0.
\]
Moreover, \eqref{eq:fixed-numerator-recovery} and $p_s\to2$ give
$Q_s(u)=\|u\|_{p_s}^2+\cD_s(u)\to\|u\|_2^2$. Using
\eqref{eq:ms-ds} and Lemma \ref{lem:distance-formula},
\[
 \dist_s(u,\cM_s)^2=Q_s(u)-m_s(u)
 \longrightarrow \|u\|_2^2-m(u)=\dist(u,\cM)^2.
\]
\end{proof}

The numerator and denominator recovery results, together with
Lemma \ref{lem:smooth-recovery-density}, give the required upper bound for the
sharp constants.

\begin{corollary}
\label{cor:sharp-recovery}
\begin{equation}\label{eq:limsup-cs}
 \limsup_{s\rightarrow0^+}c_s\leq c_{\rm LS}.
\end{equation}
\end{corollary}

\begin{proof}
Fix $v\in C^\infty(\mathbb S^n)\setminus\cM$,
\eqref{eq:fixed-distance-recovery}
implies that $\dist_s(v,\cM_s)>0$ for all sufficiently small $s$.
The definition of $c_s$ and \eqref{eq:fixed-numerator-recovery}  therefore give
\[
 \limsup_{s\rightarrow0^+}c_s
\le
 \lim_{s\rightarrow0^+}
\frac{\mathcal D_s (v)}
{2s\dist_s(v,\cM_s)^2}
=
\frac{\mathcal D(v)}
{\operatorname{dist}(v,\cM)^2}.
\]
Taking the infimum over
$v\in C^\infty(\mathbb S^n)\setminus\cM$
and using \eqref{eq:smooth-quotient-density} proves the claim.
\end{proof}

\section{The strict two-bubble bound}\label{sec:two-bubble}
We establish a strict gap below this threshold by computing the interaction of two
conformal extremals at widely separated scales. In the Euclidean case,
separation of conformal scales becomes ordinary dilation. We use the unitary Fourier transform
\[
 \widehat f(\xi)
 :=(2\pi)^{-\frac{n}{2}}
 \int_{\Rn}e^{-ix\cdot\xi}f(x)\,\dd x.
\]
Under the unitary stereographic transform, \eqref{eq:D-intro} becomes
\begin{equation}\label{eq:Euclidean-D}
 \cD_{\Rn}(f)
 =\cQ_{\Rn}(f)
 -\frac{1}{n}\int_{\Rn}f^2
 \log\frac{f^2}{\norm{f}_2^2}\,\dd x
 -\kappa_n\norm{f}_2^2,
\end{equation}
where
\[
 \cQ_{\Rn}(f)
 =\int_{\Rn}\log|\xi|\,|\widehat f(\xi)|^2\,\dd\xi,
 \qquad
 \kappa_n
 :=\log2+\psi\!\left(\frac{n}{2}\right)
 -\frac{1}{n}\log\!\left(
 \frac{\Gamma(n)}{\pi^{\frac{n}{2}}\Gamma(\frac{n}{2})}
 \right).
\]
The transform preserves the $L^2$ norm and carries $\cM$ onto
\[
 \cM_{\Rn}:=\{cB_{a,\mu}:c\in\R,\ a\in\Rn,\ \mu>0\},
 \qquad
 B_{a,\mu}(x):=\mu^{\frac{n}{2}}B(\mu(x-a)),
\]
where $B$ is as in Theorem \ref{thm:two-bubble-intro}.
The positive $L^2$-normalized optimizers are the functions $B_{a,\mu}$.
Every such optimizer $G$ satisfies
\begin{equation}\label{eq:bubble-EL-Euclidean}
 LG=\frac{1}{n}G\log G^2+\kappa_nG,
\end{equation}
where $L$ is the Fourier multiplier with symbol $\log|\xi|$.

\begin{lemma}\label{lem:exact-interaction}
Let $G,H$ be positive normalized Euclidean optimizers. Set
$ \rho :=\ip{G}{H}$,
then
\begin{equation}\label{eq:exact-interaction}
 n\cD_{\Rn}(G+H)
 =(2+2\rho)\log (2+2\rho)-2\int_{\Rn}(G+H)^2
 h\!\left(\frac{G}{G+H}\right)\dd x,
\end{equation}
where
\[
 h(t)=-t\log t-(1-t)\log(1-t),\qquad0\leq t\leq1.
\]
\end{lemma}

\begin{proof}
The definition \eqref{eq:Euclidean-D} of the deficit yields
\[
n\mathcal D_{\mathbb R^n}(G+H)=n\mathcal Q_{\mathbb R^n}(G+H)-
\int_{\mathbb R^n} (G+H)^2\log (G+H)^2\dd x+
(2+2\rho)\log (2+2\rho)-n\kappa_n (2+2\rho).
\]
Apply  \eqref{eq:bubble-EL-Euclidean} to each bubble and use the symmetry of $L$ to combine the mixed terms. Since
 $\mathcal{D}_{\mathbb R^n}(G)=\mathcal{D}_{\mathbb R^n}(H)=0$, the terms involving   $\kappa_n$ cancel,  and we obtain
\[
\begin{aligned}
n\mathcal D_{\mathbb R^n}(G+H)
&=\int_{\mathbb R^n} G^2\log G^2\dd x+\int_{\mathbb R^n} H^2\log H^2\dd x
+\int_{\mathbb R^n}
GH(\log G^2+\log H^2)\,dx\\
&\quad
-\int_{\mathbb R^n} (G+H)^2\log (G+H)^2\dd x+(2+2\rho)\log (2+2\rho).
\end{aligned}
\]
For $a,b>0$,
\[
 (a+b)^2\log(a+b)^2-a^2\log a^2-b^2\log b^2
 -ab(\log a^2+\log b^2)
 =2(a+b)^2h\!\left(\frac{a}{a+b}\right).
\]
Applying this identity pointwise with $a=G(x)$ and $b=H(x)$, and
integrating over $\mathbb R^n$, gives
\[
n\mathcal D_{\mathbb R^n}(G+H)
=
(2+2\rho)\log (2+2\rho)
-
2\int_{\mathbb R^n}
(G+H)^2
h\!\left(\frac{G}{G+H}\right)\,dx.
\]
\end{proof}
For   $B$ and $B_\lambda=B_{0,\lambda}$, 
$0<\lambda<1$, set
\begin{equation}\label{eq:rho-I}
 \rho_\lambda:=\int_{\Rn}BB_\lambda\,\dd x,
 \qquad
 I_\lambda:=2\int_{\Rn}(B+B_\lambda)^2
 h\!\left(\frac{B}{B+B_\lambda}\right)\dd x.
\end{equation}
 $\rho_\lambda$ controls the normalization term in
\eqref{eq:exact-interaction}, and $I_\lambda$ determines the leading   interaction correction.

\begin{lemma}\label{lem:radial-asymptotics}
As $\lambda\rightarrow0 ^+$,
\begin{align}
  \rho_\lambda
 &=b_n^2\abs{\mathbb S^{n-1}}\lambda^{\frac{n}{2}}\log\frac{1}{\lambda}
 +O_n(\lambda^{\frac{n}{2}}),\label{eq:rho-asymp}\\
 I_\lambda
 &=\frac n2b_n^2\abs{\mathbb S^{n-1}}
 \lambda^{\frac{n}{2}}\log^2(\frac{1}{\lambda})
 +O_n(\lambda^{\frac{n}{2}}\log\frac{1}{\lambda}).\label{eq:I-asymp}
\end{align}
\end{lemma}

\begin{proof}
We first estimate the overlap  $\rho_\lambda$.
For $r=\abs{x}$,
on $1\leq r\leq\lambda^{-1}$,
\[
 BB_\lambda r^{n-1}=b_n^2\lambda^{\frac{n}{2}}
 \frac{r^{n-1}}{(1+r^2)^{\frac{n}{2}}(1+\lambda^2r^2)^{\frac{n}{2}}}= b_n^2\lambda^{\frac{n}{2}}\frac1r
 \bigl(1+O(r^{-2}+\lambda^2r^2)\bigr).
\]
The ranges $r\leq1$ and $r\geq\lambda^{-1}$ contribute $O(\lambda^{\frac{n}{2}})$, which proves \eqref{eq:rho-asymp}.
For the entropy interaction, observe that
\[
\frac{B_\lambda(r)}{B(r)}
 =\lambda^{\frac{n}{2}}\left(\frac{1+r^2}{1+\lambda^2r^2}\right)^{\frac{n}{2}}
\]
 is increasing and equals one at $r=\lambda^{-\frac{1}{2}}$.
The weighted inversion
\[
 (\mathcal I_\lambda f)(x)
 =\lambda^{-\frac n2}|x|^{-n}
 f\!\left(\frac{x}{\lambda|x|^2}\right)
\]
interchanges $B$ and $B_\lambda$. The interaction integrand is homogeneous
of degree two in the pair of bubbles, so the square of this weight is exactly the Jacobian factor and the interaction integral is preserved. On
$\Omega_\lambda=\{r\leq\lambda^{-\frac12}\}$, the ratio lies in $(0,1]$
and satisfies
 \[
-\log \frac{B_\lambda(r)}{B(r)}=O(\log \frac{1}{\lambda}).
 \]
By direct computation, for $0<t\leq1$,
\[
 2(1+t)^2h\!\left(\frac1{1+t}\right)
 =2t(1-\log t)+O\bigl(t^2(1+\abs{\log t})\bigr).
\]
The integrated remainder is bounded by
\[
 C_n\int_{\Omega_\lambda} B_\lambda^2
 \left(1+\left|\log\frac{B_\lambda}{B}\right|\right)\dd x
 \leq C_n\lambda^{\frac n2}\log\frac1\lambda,
\]
since $B_\lambda^2\leq C_n\lambda^n$ on a region of volume
$O_n(\lambda^{-\frac n2})$. Hence
\begin{equation*}
 I_\lambda=4\int_{\Omega_\lambda}BB_\lambda
 \left(1-\log  \frac{B_\lambda(r)}{B(r)} \right)\,\dd x
 +O_n\left(\lambda^{\frac{n}{2}}\log\frac{1}{\lambda}\right).
\end{equation*}
On $r\leq 1$,
\[
\int_{r\leq 1}BB_\lambda
\left(1-\log  \frac{B_\lambda(r)}{B(r)} \right)\,\dd x
 =O_n(\lambda^{\frac{n}{2}}\log\frac{1}{\lambda}).
\]
On  $1\leq r\leq\lambda^{-\frac{1}{2}}$,
\[
 -\log \frac{B_\lambda(r)}{B(r)} =\frac n2(\log\frac{1}{\lambda}-2\log r)
 +O(r^{-2}+\lambda^2r^2).
\]
Consequently,
\begin{align*}
 I_\lambda
 &=4b_n^2\abs{\mathbb S^{n-1}}\lambda^{\frac n2}
 \int_0^{\frac{1}{2}\log\frac{1}{\lambda}}
 \left[1+\frac n2(\log\frac{1}{\lambda}-2y)\right]\dd y
 +O_n(\lambda^{\frac n2}\log\frac{1}{\lambda})\\
 &=\frac n2b_n^2\abs{\mathbb S^{n-1}}
 \lambda^{\frac n2}\log^2(\frac{1}{\lambda})
 +O_n(\lambda^{\frac n2}\log\frac{1}{\lambda}).
\end{align*}
\end{proof}
We next show
that the squared distance to the full optimizer manifold differs from one
only at the smaller overlap scale.

\begin{lemma}
As $\lambda\rightarrow0^+$,
\begin{equation}\label{eq:distance-asymptotic}
 \dist(B+B_\lambda,\cM_{\Rn})^2
 =1+O_n(\lambda^{\frac{n}{2}}\log\frac{1}{\lambda}).
\end{equation}
\end{lemma}

\begin{proof}
Set $F(t)=\langle B,B_t\rangle$. Scaling gives $F(t)=F(t^{-1})$, while
$0<F(t)\le1$ with equality only at $t=1$. A radial split over
$(0,1)$, $(1,t^{-1})$, and $(t^{-1},\infty)$ yields, for $0<t\le1$,
\[
 F(t)\le C_nt^{\frac n2}\left(1+\log\frac1t\right),
\]
and hence
\begin{equation}\label{eq:F-estimate}
 F(t)\le C_n\min\{t,t^{-1}\}^{\frac n2}(1+|\log t|),\qquad t>0.
\end{equation}
By Riesz rearrangement, scaling, and the Hilbert distance formula,
\[
 m(B+B_\lambda)^{\frac{1}{2}}
 =\sup_{\mu>0}\left[F(\mu)+F\!\left(\frac\lambda\mu\right)\right],
\quad
 \dist(B+B_\lambda,\cM_{\Rn})^2
 =2+2\rho_\lambda-m(B+B_\lambda).
\]
Testing $\mu=1$ gives the lower bound $1+\rho_\lambda$ for the
supremum. For the upper bound, write $x=\log\mu$ and
$L=\log\lambda^{-1}$. Choose $R$ so that $F(e^x)<\frac12$ for
$|x|\ge R$. The maximizer must  lie within distance $R$ of either
$0$ or $-L$; otherwise the sum is $<1$,  the test $\mu=1$ is
$>1$. In either of the two remaining regions, \eqref{eq:F-estimate}
gives
\[
 m(B+B_\lambda)^{\frac{1}{2}}
 \le 1+C_n\lambda^{\frac n2}\log\frac1\lambda.
\]
Together with \eqref{eq:rho-asymp} and
$\|B+B_\lambda\|_2^2=2+2\rho_\lambda$, this proves
\eqref{eq:distance-asymptotic}.
\end{proof}
The preceding numerator and distance asymptotics yield the strict two-bubble
inequality.

\begin{proof}[Proof of Theorem \ref{thm:two-bubble-intro}]
By Lemma \ref{lem:exact-interaction},
\[
n\mathcal D_{\mathbb R^n}(B+B_\lambda)
=
(2+2\rho_\lambda)\log (2+2\rho_\lambda)-I_\lambda.
\]
Since $\rho_\lambda\rightarrow0$,
\eqref{eq:rho-asymp} gives
\[
(2+2\rho_\lambda)\log (2+2\rho_\lambda)
=2\log2+O_n(\rho_\lambda)
=2\log2+O_n\left(\lambda^{\frac{n}{2}}\log\frac{1}{\lambda}\right).
\]
By \eqref{eq:I-asymp}, we therefore obtain
\begin{equation}\label{eq:numerator-expansion}
 \cD_{\Rn}(B+B_\lambda)
 =\frac{2\log2}{n}
 -\frac{\Gamma(n)}{\Gamma(\frac{n}{2})^2}
 \lambda^{\frac{n}{2}}\log^2(\frac{1}{\lambda})
 +O_n(\lambda^{\frac{n}{2}}\log\frac{1}{\lambda}),
\end{equation}
where we used
$\frac12b_n^2|\mathbb S^{n-1}|=
\frac{\Gamma(n)}{\Gamma(\frac{n}{2})^2}$.
Dividing \eqref{eq:numerator-expansion} by
\eqref{eq:distance-asymptotic} gives
\eqref{eq:two-bubble-main-intro}.  Since
$\lambda^{\frac{n}{2}}\log \frac{1}{\lambda}=o(\lambda^{\frac{n}{2}}\log^2(\frac{1}{\lambda}))$, the negative leading correction dominates the remainder.
Thus, for all sufficiently small $\lambda$,
\[
c_{\mathrm{LS}}
\le
\frac{\mathcal D_{\mathbb R^n}(B+B_\lambda)}
{\operatorname{dist}(B+B_\lambda,\cM_{\Rn})^2}
<
\frac{2\log2}{n},
\]
 which proves
\eqref{eq:cLS-c2b-intro}.
\end{proof}

\section{Uniform non-concentration of minimizers for $c_s$}
\label{sec:localization}
Throughout Sections \ref{sec:localization}-\ref{sec:attainment-proof}, we assume
$n\geq2$.   By a concentration-compactness
argument on the conformal cylinder, we establish uniform non-concentration of the
critical measures $\abs{u_s}^{p_s}\,d\sigma$. The strict two-bubble gap excludes dichotomy,
and the Hersch condition excludes drift of the whole mass to a pole.

Let $u_s$ be a minimizer for $c_s$, normalized by
$\norm{u_s}_{p_s}=1$.
Its minimizing identity is
\[
Q_s(u_s)-1
=
2sc_s\bigl(Q_s(u_s)-m_s(u_s)\bigr).
\]
Since the recovery bound gives $c_s=O(1)$ and
$0\leq m_s(u_s)\leq1$, we obtain
\[
1\leq Q_s(u_s)
=
\frac{1-2sc_sm_s(u_s)}{1-2sc_s}
=
1+O(s).
\]
Here $1-2sc_s>0$ by Theorem C. Moreover, by
 Corollary \ref{cor:sharp-recovery} and
\eqref{eq:cLS-c2b-intro}, there exist $\eta_0>0$ and $s_0>0$ such
that
\begin{equation}\label{eq:strict-moving-gap-intro}
 c_s
 \leq
 \frac{2-2^{\frac{n-2s}{n}}}{2s}-\eta_0,
 \qquad 0<s<s_0.
\end{equation}
The threshold on the right converges to
$\frac{2\log2}{n}$ as $s\to0^+$.
Since the stability quotient is conformally invariant, Hersch
lemma \cite{Hersch1970}, in the higher-dimensional form proved in
\cite{Laugesen2021}, allows us to choose
 a  representative satisfying
\begin{equation}\label{eq:Hersch-intro}
 \int_{\Sn}\omega\abs{u_s}^{p_s}\,d\sigma=0.
\end{equation}
\medskip

For the remainder of Sections 4-6, let
$s_k\to0^+$ and use the abbreviations
\[
 p_k:=p_{s_k},\qquad P_k:=P_{s_k},\qquad Q_k:=Q_{s_k},\qquad m_k:=m_{s_k}, \qquad
 \alpha_k:=\frac{2}{p_k}=1-\frac{2s_k}{n}.
\]
The main stability constant and optimizer manifold retain the notation
$c_{s_k}$ and $\cM_{s_k}$. After passing to subsequences when necessary, all
limits are taken along the retained subsequence without relabeling.
\medskip

The minimizing sequences below are normalized by
\[
\norm{u_k}_{p_k}=1
\]
and, when needed, by \eqref{eq:Hersch-intro}. Constants may depend on
the fixed dimension and on fixed cutoff data, but not on $k$. After
decreasing $s_0$ if necessary, we assume throughout the localization
argument that
\[
 0<s<s_0<\min\left\{1,\frac n2\right\}.
\]
The required $o(s)$ localization begins with the Aronszajn-Smith representation
for the probability measure $d\sigma$,
\begin{equation}\label{eq:AS-representation}
 Q_s(w)=\norm{w}_2^2+B_{n,s}
 \iint_{\Sn\times\Sn}
 \frac{\abs{w(\omega)-w(\eta)}^2}{\abs{\omega-\eta}^{n+2s}}
 \dd\sigma(\omega)\dd\sigma(\eta),
\end{equation}
where
\begin{equation}\label{eq:Bns}
 B_{n,s}=\frac{2^{n+2s}s\Gamma(\frac{n}{2}+1)\Gamma(\frac{n}{2}-s)}
 {\Gamma(n+1)\Gamma(1-s)}\leq C_{n,s_0}s.
\end{equation}
\begin{lemma}\label{lem:IMS}
Let $u$   belong to $\mathcal X_s$ and
 $\chi_1,\ldots,\chi_N$ be Lipschitz functions with $\sum\limits_\nu\chi_\nu^2=1$. Then
\begin{equation}\label{eq:IMS}
 \sum_\nu Q_s(\chi_\nu u)-Q_s(u)
 =B_{n,s}\iint_{\mathbb{S}^{n}\times {\mathbb{S}^{n}}} u(\omega)u(\eta)
 \frac{\Gamma_\chi(\omega,\eta)}{\abs{\omega-\eta}^{n+2s}}
 \dd\sigma(\omega)\dd\sigma(\eta),
\end{equation}
where
\[
 \Gamma_\chi(\omega,\eta)=
 \sum_\nu\abs{\chi_\nu(\omega)-\chi_\nu(\eta)}^2.
\]
\end{lemma}

\begin{proof}
Since $0<s<1$, Lipschitz multiplication preserves the form
domain of $Q_s$. Moreover,
$\Gamma_\chi(\omega,\eta)\le C|\omega-\eta|^2$,
so the error integral below is absolutely convergent.
The quadratic partition gives
$\sum_\nu\|\chi_\nu u\|_2^2=\|u\|_2^2$.
Thus the $L^2$ terms in \eqref{eq:AS-representation} cancel.
For the nonlocal terms, we have
\[
\begin{aligned}
&\sum_\nu
|\chi_\nu(\omega)u(\omega)-\chi_\nu(\eta)u(\eta)|^2
-|u(\omega)-u(\eta)|^2\\
&\qquad=
2u(\omega)u(\eta)
\left(
1-\sum_\nu\chi_\nu(\omega)\chi_\nu(\eta)
\right)=
u(\omega)u(\eta)\Gamma_\chi(\omega,\eta).
\end{aligned}
\]
Integrating this identity in \eqref{eq:AS-representation}
proves the lemma.
\end{proof}

The localization estimates proved in Appendix \ref{app:scale-localization} yield
the following concentration-compactness alternative at the $o(s_k)$ scale required here.

\begin{proposition}\label{prop:scale-cc}
Let $s_k\rightarrow0^+$, and
$\norm{u_k}_{p_k}=1$ and $Q_k(u_k)\leq1+ Cs_k$. Then there exists an infinite index set
$K\subset\mathbb N$ such that, in a fixed conformal cylinder,  exactly one of the following alternatives holds as $k\to\infty$ with $k\in K$:
\begin{enumerate}[label=\textup{(S\arabic*)},leftmargin=2.8em]
 \item there are translations $T_k$,  such that the measures $\abs{v_k(t+T_k,\theta)}^{p_k}\dd\mu$ are tight;
 \item there are finitely many localized pieces $\{u_k^j\}_{j\in J}$,
 at least two with positive limiting $L^{p_k}$-mass, such that
 \begin{align*}
  \sum_j\norm{u_k^j}_{p_{k}}^{p_{k}}&=1-o(s_k),\quad
  Q_k(u_k)=\sum_jQ_k(u_k^j)+o(s_k),\quad
  m_k(u_k)=\max_jm_k(u_k^j)+o(1),
 \end{align*}
 and the bulk supports are separated by cylinder gaps tending to infinity.
\end{enumerate}
 If $u_k$ are minimizers for $c_{s_k}$ and \eqref{eq:strict-moving-gap-intro} holds, then \textup{(S2)} is impossible.
\end{proposition}
\begin{proof}
The concentration alternative and the construction of the pieces in
\textup{(S2)} are proved in Appendix \ref{app:scale-localization}.
It remains to exclude \textup{(S2)} for minimizers for $c_{s_k}$.
All limits below are taken as $k\to\infty$, $k\in K$, with the index set supplied by the construction. Set
\[
 a_{j,k}:=\|u_k^j\|_{p_k}^{p_k},
 \qquad
 S_k:=\sum_j a_{j,k}^{\alpha_k}-1,
 \qquad
 \Delta_k:=\sum_j m_k(u_k^j)-\max_j m_k(u_k^j).
\]
The construction gives $0\le a_{j,k}\le1$,
$\sum\limits_j a_{j,k}=1-o(s_k)$, and $a_{j,k}\to a_j$, with at least two
positive $a_j$. In particular, $\sum\limits_j a_j=1$ and $\Delta_k\ge0$.

The minimizing identity and the definition of $c_{s_k}$ give
\[
\begin{aligned}
 Q_k(u_k)-1
 &=2s_kc_{s_k}\bigl(Q_k(u_k)-m_k(u_k)\bigr),\qquad
 Q_k(u_k^j)-a_{j,k}^{\alpha_k}
 \ge 2s_kc_{s_k}\bigl(Q_k(u_k^j)-m_k(u_k^j)\bigr).
\end{aligned}
\]
Summing the second inequality and using the energy and correlation
identities in \textup{(S2)}, together with $c_{s_k}=O(1)$, yields
\begin{equation}\label{eq:cc-key-upper}
 S_k\le 2s_kc_{s_k}\Delta_k+o(s_k).
\end{equation}
By H\"older inequality and the normalization of the dual weights,
\[
 0\le m_k(u_k^j)\le\|u_k^j\|_{p_k}^2=a_{j,k}^{\alpha_k}.
\]
For all sufficiently large $k$, apply Lemma \ref{lem:fractional-mixing}
to the probability vector and the corresponding nonnegative numbers
\[
 \left\{\frac{a_{j,k}}{\sum_i a_{i,k}}\right\}_j,
 \qquad
 \left\{\frac{m_k(u_k^j)}{(\sum_i a_{i,k})^{\alpha_k}}\right\}_j.
\]
Multiplying the resulting inequality by
$(\sum_i a_{i,k})^{\alpha_k}=1+o(s_k)$ gives
\begin{equation}\label{eq:cc-key-lower}
\begin{aligned}
 S_k
 &\ge (2-2^{\alpha_k})\Delta_k
      +\left(\sum_j a_{j,k}\right)^{\alpha_k}-1= (2-2^{\alpha_k})\Delta_k+o(s_k).
\end{aligned}
\end{equation}
The uniform gap \eqref{eq:strict-moving-gap-intro}
 implies
$2-2^{\alpha_k}-2s_kc_{s_k}\ge2s_k\eta_0$.
Combining \eqref{eq:cc-key-upper} and \eqref{eq:cc-key-lower}, we obtain
$\Delta_k=o(1)$, and hence
\[
 \limsup_{\substack{k\to\infty\\ k\in K}}\frac{S_k}{s_k}\le0.
\]
On the other hand,
$x^{\alpha_k}-x\ge(1-\alpha_k)x\log\frac1x$ for $0\le x\le1$,
with the value at $x=0$ defined by continuity. Therefore
\[
 S_k\ge\sum_j a_{j,k}-1
       +\frac{2s_k}{n}\sum_j a_{j,k}\log\frac1{a_{j,k}},
\]
and the finite number of pieces gives
\begin{equation}\label{eq:cc-positive-mixing}
 \liminf_{\substack{k\to\infty\\ k\in K}}\frac{S_k}{s_k}
 \ge-\frac2n\sum_j a_j\log a_j>0.
\end{equation}
The strict positivity follows from $\sum\limits_j a_j=1$ and the presence of
at least two positive limiting masses. This contradicts the preceding
upper bound and excludes \textup{(S2)} for minimizers.
\end{proof}

\begin{corollary}\label{cor:Hersch-no-atoms}
Assume the hypotheses of Proposition \ref{prop:scale-cc},
and suppose that $u_k$ are normalized minimizers  for $c_{s_k}$
satisfying the strict gap
\eqref{eq:strict-moving-gap-intro}
and the Hersch condition
\eqref{eq:Hersch-intro}.
Then the probability measures
$ \mu_k=|u_k|^{p_k}\,d\sigma$
satisfy the uniform non-concentration estimate
\begin{equation}\label{eq:uniform-no-atoms}
 \lim_{r\rightarrow0^+}
 \limsup_{k\to\infty}
 \sup_{\xi\in\mathbb S^n}
 \mu_k(B_r(\xi))
 =0.
\end{equation}
\end{corollary}

\begin{proof}
Fix a pair of antipodal poles and the associated conformal cylinder.
For every infinite index set $K_0\subset\mathbb N$, Proposition
\ref{prop:scale-cc}, together with the strict gap
\eqref{eq:strict-moving-gap-intro}, yields an infinite
$K_1\subset K_0$ and translations $T_k$ satisfying
\textup{(S1)}, the alternative \textup{(S2)} is excluded.
The translations $T_k$ are bounded on $K_1$. Otherwise, after passing
to $K_2\subset K_1$, either $T_k\to+\infty$ or $T_k\to-\infty$.
Tightness in the translated coordinates would then imply
$\mu_k\rightharpoonup\delta_p$,
where $p$ is the corresponding pole, contradicting
\[
0=\int_{\mathbb S^n}\omega\,d\mu_k(\omega)
\longrightarrow p\neq0.
\]
Thus, after a further extraction, the measures are tight in the
unshifted cylinder.

Suppose now that \eqref{eq:uniform-no-atoms} fails. Then there exist $\eta>0$,  an infinite set $K_0\subset\mathbb N$, radii $r_k\to0$ and
points $\xi_k\in\mathbb S^n$,
$k\in K_0$, such that
\[
\mu_k(B_{r_k}(\xi_k))\ge\eta.
\]
By compactness of $\mathbb S^n$, there exist an
infinite $K_1\subset K_0$ and $\xi \in\mathbb S^n$ such that
 $\xi_k\to\xi$ as $k \rightarrow\infty$, $k \in K_1$.
 Set
$\rho_k:=r_k+d(\xi_k,\xi)$,
so that $\rho_k\to0$ and
$B_{r_k}(\xi_k)\subset B_{\rho_k}(\xi)$.
Apply the preceding tightness conclusion to the sequence indexed by   $K_1$, using  the fixed cylinder
with poles $\xi$ and $-\xi$.
There is an
infinite set $K_2\subset K_1$ along which the measures are tight  in this cylinder.
With  $\xi$ corresponding to $t=+\infty$, choose
$R_k\to\infty$,   so that
$B_{\rho_k}(\xi)\subset\{t\ge R_k\}$. Tightness gives
$\mu_k(\{t\ge R_k\})\to0$, contradicting
the lower bound $\eta$.
Thus \eqref{eq:uniform-no-atoms} holds.
\end{proof}

\section{ Nonvanishing, active-set equations, and uniform boundedness}
\label{sec:uniform-boundedness}
 We first establish nonvanishing for nonnegative
low-energy sequences, a sign-mixing estimate then excludes cancellation for the
signed minimizers. The maxima may still blow up, we control this possibility through the
active-set Euler-Lagrange equation.
\medskip

\begin{lemma}\label{lem:nonnegative-nonvanishing}
Let $s_k\to0^+$ and let $w_k\ge0$ belong to the form domain of $Q_k$.
Suppose that
\[
 \|w_k\|_{p_k}=1,\qquad Q_k(w_k)\le1+C_0s_k,
\]
 and that $w_k^{p_k}\,d\sigma$ satisfies
\eqref{eq:uniform-no-atoms}. Then
\begin{equation}\label{eq:no-L2-collapse}
 \liminf_{k\to\infty}\|w_k\|_2^2>0.
\end{equation}
Moreover, if $K\subset\mathbb N$ is infinite and
$w_k\rightharpoonup w$ in $L^2(\Sn)$ as $k\to\infty$, $k\in K$,
then $w\ne0$.
\end{lemma}

\begin{proof}
Suppose first that $\|w_k\|_2\to0$ along a subsequence. Fix
$0<\varepsilon<1$. Uniform non-concentration provides a fixed radius
$r>0$ and  $\sum_j\chi_j^2=1$, with
bounded overlap $N_0$, such that each support carries at most
$\varepsilon$ of the $L^{p_k}$ mass. \eqref{eq:IMS}
 and
$B_{n,s_k}=O(s_k)$ give
\[
 \sum_jQ_k(\chi_jw_k)=Q_k(w_k)+o(s_k).
\]
Writing $a_{j,k}=\|\chi_jw_k\|_{p_k}^{p_k}$, bounded overlap and the
sharp Sobolev inequality imply
\[
 \varepsilon^{-\frac{2s_k}{n}}N_0^{1-\frac{p_k}{2}}
 \le \sum_ja_{j,k}^{\alpha_k}
 \le 1+C_0s_k+o(s_k).
\]
After taking logarithms and dividing by $s_k$,
$\frac2n\log\frac1{N_0\varepsilon}\le C_0$, a contradiction as
$\varepsilon\rightarrow0^+$. Thus \eqref{eq:no-L2-collapse} holds.

Now suppose $w_k\rightharpoonup0$. Since $w_k\ge0$, 
$\|w_k\|_1\to0$, while \eqref{eq:no-L2-collapse} gives
$\|w_k\|_2\ge z>0$. Interpolation yields
$\|w_k\|_1^{-2s_k/n}\le C_z$. Let
$\rho_k=\|w_k\|_1^{1/n}$. Repeating the preceding cap localization at
scale $\rho_k$, the near-far estimate in \eqref{eq:IMS} gives an
$O(s_k)$ localization error. Set
\[
 \varepsilon_k:=\sup_{\xi\in\Sn}
 \int_{B_{C_n\rho_k}(\xi)}w_k^{p_k}\,d\sigma.
\]
Uniform non-concentration gives $\varepsilon_k\to0$. And the
same mass-power estimate yields
\[
 \varepsilon_k^{-\frac{2s_k}{n}}N_0^{1-\frac{p_k}{2}}
 \le1+C s_k.
\]
Taking logarithms and dividing by $s_k$ bounds
$\log(1/\varepsilon_k)$ uniformly, a contradiction. Hence every weak
limit is nonzero.
\end{proof}
\medskip

The preceding lemma prevents $|u_k|$ from vanishing, but it does not exclude
cancellation between the positive and negative parts of $u_k$. The following energy
estimate rules out this possibility.

\begin{lemma} \label{lem:sign-mixing}
Let $s_k\to0^+$, with $0<s_k<\min\{1,\frac{n}{2}\}$, and let
$u_k\in H^{s_k}(\mathbb S^n)$.
Assume that $u_k\rightharpoonup u$ and
$|u_k|\rightharpoonup v$ in $L^2(\mathbb S^n)$.
  If $v\neq\abs u$, then
\begin{equation}\label{eq:sign-mixing}
 \frac{Q_k(u_k)-Q_k(\abs{u_k})}{s_k}\longrightarrow+\infty.
\end{equation}
\end{lemma}

\begin{proof}
 Write $u_{k,\pm}=\frac{|u_k|\pm u_k}{2}$.
Then
\[
u_{k,+}\rightharpoonup
a:=\frac{v+u}{2},
\qquad  u_{k,-}\rightharpoonup
b:=\frac{v-u}{2}
\]
in $L^2(\mathbb S^n)$. The positive cone is weakly closed, hence
$a,b\ge0$.  Moreover,
\[
v-|u|
=
a+b-|a-b|
=
2\min\{a,b\}.
\]
Thus $v\neq|u|$ implies that $a(\omega)b(\omega)>0$ on a set of
positive measure. Using \eqref{eq:AS-representation} and symmetry, we deduce that
\begin{equation*}
 Q_k(u_k)-Q_k(\abs{u_k})
 =
 8B_{n,s_k}
 \iint_{\Sn\times\Sn}
 \frac{u_{k,+}(\omega)u_{k,-}(\eta)}
 {\abs{\omega-\eta}^{n+2s_k}}
 \dd\sigma(\omega)\dd\sigma(\eta).
\end{equation*}
For $j\geq1$, let $K_j$ be the integral operator with kernel
 $\one_{\{2^{-j-1}<\abs{\omega-\eta}\leq2^{-j}\}}
 \abs{\omega-\eta}^{-n}$.
For fixed $j$, this kernel belongs to
$L^2(\mathbb S^n\times\mathbb S^n)$, so $K_j$ is
 compact on $L^2$.
 Hence
\[
K_j u_{k,-}\longrightarrow K_jb
\quad\text{strongly in }L^2,
\qquad
\langle u_{k,+},K_j u_{k,-}\rangle
\longrightarrow\langle a,K_jb\rangle.
\]
Since $|\omega-\eta|^{-n-2s_k}\ge|\omega-\eta|^{-n}$
on these annuli and  $\frac{B_{n,s}}{s}\to b_n^*>0$,
 for every fixed $J$,
\[
 \liminf_{k\to\infty}
 \frac{Q_k(u_k)-Q_k(\abs{u_k})}{s_k}
 \geq
 8b_n^*\sum_{j=1}^J\ip{a}{K_jb}.
\]
It remains to show that the series on the right diverges.
For almost every $\omega\in E:=\{\omega\in\mathbb S^n:a(\omega)b(\omega)>0\}$,
\[
K_jb(\omega)\ge c_n b(\omega)>0
\qquad\text{for all sufficiently large }j,
\]
we have $\sum\limits_{j=1}^\infty K_jb(\omega)=+\infty$ a.e.
on $E$. Since $a>0$ on $E$ and $\sigma(E)>0$,
Tonelli theorem gives
\[
\sum_{j=1}^\infty\langle a,K_jb\rangle=+\infty.
\]
Letting $J\to\infty$ completes the proof.
\end{proof}
\medskip

Combining nonvanishing of $|u_k|$ with the sign-mixing estimate yields a nonzero
weak limit of the minimizers.
\begin{corollary}
\label{cor:weak-nonvanishing}
Let $s_k\to0^+$ and let $u_k$ be the Hersch-normalized minimizers for $c_{s_k}$ with $\|u_k\|_{p_k}=1$.
There exists an infinite set $K\subset \mathbb{N}$ and
$u\in L^2(\Sn)\setminus\{0\}$ such that
\begin{equation}\label{eq:weak-nonzero}
 u_k\weak u
 \quad\text{in }L^2(\Sn)  \quad \text{as} \quad k\rightarrow \infty, \  k\in K.
\end{equation}
\end{corollary}

\begin{proof}
Since $\|u_k\|_2\le \|u_k\|_{p_k}=1$, weak compactness gives
an infinite set $K\subset\mathbb N$,  
and $u,v\in L^2(\mathbb S^n)$ such that, as $ k\rightarrow \infty$, 
\[
u_k\rightharpoonup u,
\qquad
|u_k|\rightharpoonup v
\quad\text{in }L^2(\mathbb S^n) .
\]
The representation \eqref{eq:AS-representation},  the sharp Sobolev
inequality and the preceding energy estimate for normalized minimizers
give
\[
1\le Q_k(|u_k|)\le Q_k(u_k)\le 1+C_0 s_k.
\]The normalization and critical mass measures are unchanged by taking absolute values. Therefore,
Lemma \ref{lem:nonnegative-nonvanishing}
gives $v\neq0$.
Moreover,
\[
0\le
\frac{Q_k(u_k)-Q_k(|u_k|)}{s_k}
\le C_0.
\]
Lemma \ref{lem:sign-mixing} forces $v=|u|$,
 and hence $u\neq0$.
\end{proof}
\medskip

Next we use the Euler-Lagrange equation to control the maxima.
The nearest bubble need not be unique, so the distance term
may not be differentiable. We therefore retain the full active-set.
We use the duality map $\mathfrak g_s$ defined in
\eqref{eq:signed-dual-map}. For every
$h\in\cM_{s,1}$, one has
$P_sh=\mathfrak g_s(h)$ by the Euler-Lagrange equation above.
For $u\ne0$, set
\[
 \mathcal A_s(u)=\left\{h\in\cM_{s,1}:
 \left|\int_{\mathbb{S}^n} u\,\mathfrak g_s(h)\,\dd\sigma\right|^2=m_s(u)\right\}.
\]
This active-set is nonempty and compact.  Indeed, if all correlations with
the positive dual weights $g_{\xi,s}=h_{\xi,s}^{p_s-1}$ vanished, analyticity
of $r\mapsto g_{r\theta,s}$ at $r=0$, followed by polarization and density of
spherical polynomials, would force $u=0$.  Thus $m_s(u)>0$.  On the other
hand, if $|\xi_k|\to1$,  H\"older inequality and absolute continuity of
the $L^{p_s}$ integral give
\[
\int_{\mathbb{S}^n} u g_{\xi_k,s}\,\dd\sigma\longrightarrow0.
\]
Hence a maximizing sequence for a positive correlation cannot escape to the  boundary of the parameter ball.
This compactness also permits the use of the Clarke-Danskin maximum rule in a
neighborhood of $u_s$ \cite{Clarke1983,Danskin1967}.

\begin{proposition}\label{prop:dual-critical}
Fix $0<s<\frac{n}{2}$.
Let $u_s\in H^s(\mathbb S^n)
\setminus\cM_s$ attain $c_s$, with $\|u_s\|_{p_s}=1$.
 There is a probability measure $\nu_s$ on $\mathcal A_s(u_s)$ such that
\begin{equation}\label{eq:EL-full}
 (1-2sc_s)P_su_s
 =\abs{u_s}^{p_s-2}u_s-2sc_sW_s,
\end{equation}
where
\begin{equation}\label{eq:Ws}
 W_s:=\int_{\mathcal A_s(u_s)}a_s(h)\,\mathfrak g_s(h)\,\dd\nu_s(h),
 \qquad
 a_s(h):=\int_{\mathbb S^n} u_s\,\mathfrak g_s(h)\,\dd\sigma.
\end{equation}
Moreover,
\begin{equation}\label{eq:residual-bound}
 \norm{P_su_s-\abs{u_s}^{p_s-2}u_s}_{H^{-s}_{P_s}}
 \leq2sc_s \dist_s(u_s,\cM_s),
\end{equation}
where $\norm{f}_{H^{-s}_{P_s}}^2:=\ip{f}{P_s^{-1}f}$. If the active bubble is unique up to sign, equality holds in \eqref{eq:residual-bound}.
\end{proposition}
\begin{proof}
By the definition of $c_s$,
\[
 J_s(v):=\cD_s(v)-2sc_s\bigl(Q_s(v)-m_s(v)\bigr)\ge0,
 \qquad J_s(u_s)=0,
\]
so $u_s$ is a global minimizer of $J_s$. For
$\ell_h(u)=\int u\,\mathfrak g_s(h)\,d\sigma$, the boundary-correlation
estimate above allows the supremum
$m_s(v)=\sup_h\ell_h(v)^2$ to be restricted, near $u_s$, to a fixed
compact parameter set. Moreover,
\[
 D(\ell_h^2)(u_s)[\varphi]
 =2a_s(h)\int_{\Sn}\varphi\,\mathfrak g_s(h)\,d\sigma.
\]
The Clarke maximum rule \cite{Clarke1983}
and the Fermat condition $0\in\partial^CJ_s(u_s)$ give a
probability measure $\nu_s$ on $\mathcal A_s(u_s)$ such that  $W_s$ in \eqref{eq:Ws} satisfies \eqref{eq:EL-full} in
$H^{-s}_{P_s}$.

For every active $h$,
$P_sh=\mathfrak g_s(h)$, $Q_s(h)=1$, and $a_s(h)^2=m_s(u_s)$. Hence
\[
 \|P_su_s-a_s(h)\mathfrak g_s(h)\|_{H^{-s}_{P_s}}^2
 =Q_s(u_s)-m_s(u_s)=\dist_s(u_s,\cM_s)^2.
\]
Averaging in $h$ and using convexity gives
$\|P_su_s-W_s\|_{H^{-s}_{P_s}}\le\dist_s(u_s,\cM_s)$. Equation
\eqref{eq:EL-full} now yields \eqref{eq:residual-bound}. If the active
bubble is unique up to sign, the barycenter is a single vector, and
equality holds.
\end{proof}

To control the forcing term uniformly, it remains to prevent active bubbles from
escaping to the boundary of the conformal parameter space. Along a subsequence for
which $u_k\rightharpoonup u\neq0$, this follows from the next estimate.

\begin{lemma}\label{lem:escaping-corr-zero}
Let $s_k\to0^+$ and assume that $\|u_k\|_{p_k}=1$ and
\eqref{eq:uniform-no-atoms} holds. If normalized fractional bubbles $h_k$ satisfy $\abs{h_k}^{p_k}\dd\sigma \weak \delta_\xi$, then
\begin{equation}\label{eq:escaping-corr-zero}
\int_{\mathbb{S}^n}u_k\abs{h_k}^{p_k-2}h_k\,\dd\sigma
\longrightarrow0.
\end{equation}
\end{lemma}
\begin{proof}
Define the probability measures
\[
d\rho_k=|u_k|^{p_k}\,d\sigma,
\qquad
d\eta_k=|h_k|^{p_k}\,d\sigma.
\]
H\"older inequality on $B_r(\xi)$ and its complement gives
\[
 \left|\int_{\mathbb{S}^n} u_k\abs{h_k}^{p_k-2}h_k \,d\sigma\right|
 \leq\rho_k(B_r(\xi))^{\frac{1}{p_k}}
 +\eta_k(B_r(\xi)^c)^{1-\frac{1}{p_k}}.
\]
For fixed $r>0$, weak convergence
$\eta_k\rightharpoonup\delta_\xi$ gives
\[
\eta_k(B_r(\xi)^c)\longrightarrow0.
\]
Moreover, since $p_{k}\to2$,   $\frac{1}{p_{k}}$ and $1-\frac{1}{p_{k}}$ are at least $\frac{1}{3}$
for all sufficiently large $k$.
Hence
\[
\limsup_{k\to\infty}
\left|
\int_{\mathbb{S}^n} u_k|h_k|^{p_{k}-2}h_k\,d\sigma
\right|
\le
\left(
\limsup_{k\to\infty}
\sup_{\zeta\in\mathbb S^n}
\rho_k(B_r(\zeta))
\right)^{\frac{1}{3}}.
\]
Letting $r\rightarrow0^+$ proves \eqref{eq:escaping-corr-zero}.
\end{proof}

\begin{proposition}\label{prop:correlation-compactness}
Let $K\subset\mathbb N$ be infinite and let $s_k\to0^+$ along $K$.
Suppose that $\|u_k\|_{p_k}=1$, that the associated measures satisfy
\eqref{eq:uniform-no-atoms} with $k\in K$, and that
\[
 u_k\rightharpoonup u\ne0
 \quad\text{in }L^2(\Sn)
 \quad\text{as }k\to\infty,\ k\in K.
\]
Then
\begin{equation}\label{eq:m-convergence}
 m_k(u_k)\longrightarrow m(u)>0
 \quad\text{as }k\to\infty,\ k\in K.
\end{equation}
Moreover, there exist $r\in(0,1)$ and $k_0$ such that
\[
 \mathcal A_{s_k}(u_k)
 \subset\{\pm h_{\xi,s_k}:|\xi|\le r\},
 \qquad k\in K,\ k\ge k_0.
\]
If $W_k$ is the active-set average in \eqref{eq:Ws},  then, for every integer $q\ge0$,
\begin{equation}\label{eq:W-Cq}
 \sup_{\substack{k\in K\\ k\ge k_0}}\|W_k\|_{C^q(\Sn)}\le C_q.
\end{equation}
\end{proposition}

\begin{proof}
First, $m(u)>0$. Otherwise,
\[
\int_{\Sn}u(\omega)(1-\xi\cdot\omega)^{-\frac{n}{2}}\,d\sigma(\omega)
\]
vanishes on $\mathbb B^{n+1}$. Differentiating at $\xi=0$ gives
$\int_{\Sn}u(\omega)\omega^\alpha\,d\sigma(\omega)=0$
for every multi-index $\alpha$. Since spherical polynomials are dense in
$L^2(\Sn)$, this implies $u=0$, a contradiction.

We next prove \eqref{eq:m-convergence}. Choose a subsequence $K_1\subset K$
realizing the upper limit of $m_k(u_k)^{\frac{1}{2}}$. For $k\in K_1$, choose
$\xi_k\in\mathbb B^{n+1}$ such that
\[
m_k(u_k)^{\frac{1}{2}}
\leq
|\langle u_k,g_{\xi_k,s_k}\rangle|+\frac1k.
\]
After passing to a further subsequence, $\xi_k\to\bar\xi\in
\overline{\mathbb B^{n+1}}$.
If $|\bar\xi|<1$, then
$g_{\xi_k,s_k}\to h_{\bar\xi,0}$
strongly in $L^2(\Sn)$,
and hence
\[
|\langle u_k,g_{\xi_k,s_k}\rangle|
\to
|\langle u,h_{\bar\xi,0}\rangle|
\leq m(u)^{\frac{1}{2}}.
\]
If $|\bar\xi|=1$, the corresponding critical bubble measures
concentrate at $\bar\xi$, and 
Lemma \ref{lem:escaping-corr-zero} gives
$\langle u_k,g_{\xi_k,s_k}\rangle\to0$.
Therefore,
\[
\limsup_{\substack{k\to\infty\\k\in K}}
m_k(u_k)^{\frac{1}{2}}
\leq m(u)^{\frac{1}{2}}.
\]
Conversely, for every fixed $\xi\in\mathbb B^{n+1}$,
$g_{\xi,s_k}\to h_{\xi,0}$
strongly in $L^2(\Sn)$.
Thus
\[
\liminf_{\substack{k\to\infty\\k\in K}}
m_k(u_k)^{\frac{1}{2}}
\geq
|\langle u,h_{\xi,0}\rangle|.
\]
Taking the supremum over $\xi$ 
and combining the resulting lower bound
with the preceding upper bound
proves \eqref{eq:m-convergence}.

It remains to control the active parameters. Suppose, for contradiction, that active parameters escape to the boundary
of $\mathbb B^{n+1}$. After a further extraction,
$\xi_k\to\xi\in\partial\mathbb B^{n+1}$.
By Lemma \ref{lem:escaping-corr-zero},
$\langle u_k,\mathfrak g_{s_k}(h_k)\rangle\to0$, contradicting
\[
|\langle u_k,\mathfrak g_{s_k}(h_k)\rangle|^2
=m_k(u_k)\to m(u)>0.
\] This proves the compactness of all active parameters.
On this compact parameter set,  
\eqref{eq:fractional-dual-weight} gives
\[
\sup_{\substack{k\in K\\k\ge k_0}}
\sup_{h\in\mathcal A_{s_k}(u_k)}
\|\mathfrak g_{s_k}(h)\|_{C^q(\Sn)}
\leq C_q.
\]
By Hölder inequality and \eqref{eq:Ws}, 
\[
\|W_k\|_{C^q(\Sn)}
\le
\int_{\mathcal A_{s_k}(u_k)}
|a_k(h)|
\|\mathfrak g_{s_k}(h)\|_{C^q(\Sn)}
\,d\nu_k(h)
\le C_q.
\]
This proves \eqref{eq:W-Cq}.
\end{proof}

With $W_k$ uniformly controlled, fixed-order regularity provides a maximum point,
while uniform control of that maximum requires the following  estimates.

\begin{lemma}
\label{lem:kernel-constants}
Let $n\ge2$, $0<s<s_0<1$. Set
$\kappa_{n,s}
:=\frac{|\mathbb S^n|\Gamma(\frac n2+s)}
{2^{2s}\pi^{n/2}\Gamma(s)}$.
Then, for every $g\in L^{p_s'}(\mathbb S^n)$,
\begin{equation}\label{eq:inverse-kernel}
 P_s^{-1}g
 = \kappa_{n,s}\int_{\Sn}
 \frac{g(\eta)}{\abs{\omega-\eta}^{n-2s}}\dd\sigma(\eta)
 \qquad\text{in }L^{p_s}(\Sn),
\end{equation}
 hence almost everywhere.
For  $0<R\le1$,
\begin{equation}\label{eq:kernel-near}
 \kappa_{n,s}\int_{\abs{\omega-\eta}<R}
 \abs{\omega-\eta}^{2s-n}\dd\sigma(\eta)
 \leq A_{n,s}R^{2s},
\end{equation}
where
\begin{equation}\label{eq:A-ns}
 A_{n,s}
 :=
 \frac{\Gamma(\frac{n}{2}+s)}
 {2^{2s}\Gamma(\frac{n}{2})\Gamma(1+s)}
 =
 1+O_n(s).
\end{equation}
Moreover,
\begin{equation}\label{eq:kernel-far}
 \kappa_{n,s}
 \left(
 \int_{\abs{\omega-\eta}>R}
 \abs{\omega-\eta}^{-(n-2s)p_s}\dd\sigma(\eta)
 \right)^{\frac{1}{p_s}}
 \leq  C_ns R^{-\frac{n}{p_s}}.
\end{equation}
All constants are uniform for $0<s<s_0$.
\end{lemma}

The proof is given in Appendix \ref{app:regularity}.

\begin{proposition}
 Let $s_k\to0^+$ with $0<s_k<\min\{1,\frac n6\}$, and let $u_k$ be
Hersch-normalized minimizers for $c_{s_k}$ satisfying $\|u_k\|_{p_k}=1$.
Then
\begin{equation}\label{eq:uniform-Linf}
 \sup_k\|u_k\|_\infty<\infty.
\end{equation}   
\end{proposition}

\begin{proof}
For each fixed $k$, Lemma \ref{lem:fixed-regularity} gives a
continuous representative of $u_k$.
Both sides of the inverted Euler-Lagrange equation are continuous,
so the equation holds pointwise.
Suppose that \eqref{eq:uniform-Linf} fails. There is an infinite set
$K_0\subset\mathbb N$ such that
\[
 M_k:=\|u_k\|_\infty\longrightarrow\infty
 \quad\text{as }k\to\infty,\ k\in K_0.
\]
By Corollary \ref{cor:Hersch-no-atoms}, the critical measures
satisfy \eqref{eq:uniform-no-atoms}.
Corollary \ref{cor:weak-nonvanishing} gives an infinite
$K\subset K_0$ along which $u_k$ has a nonzero weak $L^2$ limit.
Proposition \ref{prop:correlation-compactness} then yields
\[
 \|W_k\|_\infty\le C_W.
\]
All subsequent estimates are for sufficiently large $k\in K$.

Choose $\omega_k\in\Sn$ with $|u_k(\omega_k)|=M_k$, and define
\begin{equation}\label{eq:natural-radius}
 r_k:=M_k^{-\frac{p_k}{n}}
     =M_k^{-\frac{2}{n-2s_k}}.
\end{equation}
In this proof, write
$B_r^{\mathrm{ch}}(\omega):=\{\eta\in\Sn:|\omega-\eta|<r\}$ and set
\[
 a_k:=\int_{B_{r_k}^{\mathrm{ch}}(\omega_k)}|u_k|^{p_k}\,d\sigma.
\]
Normalization and continuity give $M_k\ge1$, $r_k\le1$, and
$0<a_k\le1$. Splitting the nonlinear potential over
$B_{a_k^{\frac1n}r_k}^{\mathrm{ch}}(\omega_k)$,
$B_{r_k}^{\mathrm{ch}}(\omega_k)\setminus
 B_{a_k^{\frac1n}r_k}^{\mathrm{ch}}(\omega_k)$, and
$B_{r_k}^{\mathrm{ch}}(\omega_k)^c$, H\"older inequality and
\eqref{eq:kernel-near}-\eqref{eq:kernel-far} give
\[
\begin{aligned}
 P_k^{-1}(|u_k|^{p_k-1})(\omega_k)
 &\le A_{n,s_k}M_k^{p_k-1}(a_k^{\frac1n}r_k)^{2s_k}+C_ns_k\left(
 a_k^{\frac1{p_k'}}(a_k^{\frac1n}r_k)^{-\frac n{p_k}}
 +r_k^{-\frac n{p_k}}\right).
\end{aligned}
\]
The identities $p_k(1-\frac{2s_k}{n})=2$ and
$\frac1{p_k'}-\frac1{p_k}=\frac{2s_k}{n}$ yield
\[
\begin{gathered}
 M_k^{p_k-1}(a_k^{\frac1n}r_k)^{2s_k}
 =a_k^{\frac1{p_k'}}(a_k^{\frac1n}r_k)^{-\frac n{p_k}}
 =M_ka_k^{\frac{2s_k}{n}}, \qquad
 r_k^{-\frac n{p_k}}=M_k.
\end{gathered}
\]
Since $A_{n,s_k}=1+O_n(s_k)$,
\[
 P_k^{-1}(|u_k|^{p_k-1})(\omega_k)
 \le(1+C_ns_k)M_ka_k^{\frac{2s_k}{n}}+C_ns_kM_k.
\]
Positivity of the inverse kernel and $P_k^{-1}\one=\one$ imply
$\|P_k^{-1}W_k\|_\infty\le C_W$. Evaluating the inverted
Euler-Lagrange equation at $\omega_k$ therefore gives
\[
\begin{aligned}
 (1-2s_kc_{s_k})M_k
 &\le P_k^{-1}(|u_k|^{p_k-1})(\omega_k)
       +2s_kc_{s_k}\|P_k^{-1}W_k\|_\infty\\
 &\le(1+C_ns_k)M_ka_k^{\frac{2s_k}{n}}
       +C_ns_kM_k+2s_kc_{s_k}C_W.
\end{aligned}
\]
Divide by $M_k\ge1$ and use $c_{s_k}=O(1)$. For a constant $C$
independent of $k\in K$,
\[
 a_k^{\frac{2s_k}{n}}\ge1-Cs_k.
\]
For all sufficiently large $k\in K$, $Cs_k\le\frac12$, so
\[
 \log a_k\ge\frac n{2s_k}\log(1-Cs_k)\ge-nC.
\]
Consequently, there is $\eta>0$, independent of $k\in K$, such that
\begin{equation*}
 \int_{B_{r_k}^{\mathrm{ch}}(\omega_k)}|u_k|^{p_k}\,d\sigma\ge\eta
\end{equation*}
for all sufficiently large $k\in K$.
Since $M_k\to\infty$ along $K$, \eqref{eq:natural-radius} gives
$r_k\to0$. For every fixed $\rho>0$, comparison of chordal and
geodesic distances yields
\[
 B_{r_k}^{\mathrm{ch}}(\omega_k)
 \subset B_{2r_k}(\omega_k)\subset B_\rho(\omega_k)
\]
for all sufficiently large $k\in K$. Thus
\[
 \eta\le
 \limsup_{\substack{k\to\infty\\k\in K}}
 \sup_{\xi\in\Sn}\int_{B_\rho(\xi)}|u_k|^{p_k}\,d\sigma.
\]
Letting $\rho\to0^+$ contradicts Corollary \ref{cor:Hersch-no-atoms}.
This proves \eqref{eq:uniform-Linf}.
\end{proof}

\section{Proof of Theorem \ref{thm:main}}
\label{sec:attainment-proof}
The uniform $L^\infty$ bound restores strong $L^2$ compactness. It remains to pass
to the lower limit in  deficit of the fractional
Sobolev inequality
and to exclude convergence to the
optimizer manifold.

\begin{proposition}\label{prop:strong-L2}
Let $K\subset\mathbb N$ and $u\neq0$ be as in
Corollary \ref{cor:weak-nonvanishing} for the Hersch-normalized
sequence of normalized minimizers $\{u_k\}$. Then, as
$k\to\infty$, $k\in K$,
\[
u_k\longrightarrow u
\quad\text{strongly in }L^2(\Sn),
\qquad
\|u\|_2=1.
\]
Moreover,
\begin{equation}\label{eq:distance-limit}
\dist_{s_k}(u_k,\cM_{s_k})^2
\longrightarrow
\dist(u,\cM)^2.
\end{equation}
\end{proposition}
\begin{proof}
The uniform $L^\infty$ bound gives $\sup_k\|u_k\|_{2+\varepsilon}<\infty$
for some fixed $\varepsilon>0$. Since $p_k\to2$ and
$\|u_k\|_{p_k}=1$, the identity
$x^{p_k/2}-x=\int_1^{p_k/2}x^t\log x\,dt$, together with the
$L^{2+\varepsilon}$ bound, gives
\[
 1-\|u_k\|_2^2=O(s_k),\qquad \|u_k\|_2\to1.
\]
As $Q_k(u_k)\le1+Cs_k$, Lemma \ref{lem:spectral-domination} yields
\[
 \langle u_k,\cH u_k\rangle
 \le\frac{Q_k(u_k)-\|u_k\|_2^2}{2s_k}\le C.
\]
The compact embedding $\cX\hookrightarrow L^2$ and the weak limit  imply $u_k\to u$ strongly in $L^2$ and $\|u\|_2=1$.
Finally, $Q_k(u_k)\to1$ and Proposition
\ref{prop:correlation-compactness} gives $m_k(u_k)\to m(u)$, hence
\[
 \dist_{s_k}(u_k,\cM_{s_k})^2
 =Q_k(u_k)-m_k(u_k)
 \longrightarrow1-m(u)=\dist(u,\cM)^2.
\]
\end{proof}

\begin{proposition}
\label{prop:numerator-liminf}
Let $s_k\rightarrow 0^+$ and $p_k=\frac{2n}{n-2s_k}$.
Suppose that $\{u_k\}$ satisfies
\[
\|u_k\|_{p_k}=1,
\qquad
Q_k(u_k)\le1+Ks_k,
\qquad
\sup_k\|u_k\|_\infty<\infty,
\]
and that
$u_k\rightarrow u$  strongly in $L^2(\mathbb S^n)$,
$\|u\|_2=1$.
Then $u\in\mathcal X$ and
\[
\liminf_{k\to\infty}\frac{1}{2s_{k}}
\mathcal D_{s_k} (u_k)
\ge \mathcal D(u).
\]
\end{proposition}

\begin{proof}
Strong $L^2$ convergence and the uniform $L^\infty$ bound give
$|u_k|^2\to |u|^2$ in $L^1$ and $0\le |u_k|^2,|u|^2\le C$ for some fixed
$C\ge1$.  Consequently,
$|u_k|^2\log |u_k|^2\to |u|^2\log |u|^2$ in $L^1$, with $0\log0=0$.
Uniformly for $0\le x\le C$ and $0\le\delta\le1$,
\[
x^{1+\delta}
=
x+\delta x\log x+O(\delta^2).
\]
Using $\|u_k\|_{p_k}=1$, we obtain
\begin{equation}\label{eq:entropy-moving}
\begin{aligned}
\frac{1-\|u_k\|_2^2}{2s_k}
&=
\frac1{n-2s_k}
\int_{\mathbb S^n}|u_k|^2\log |u_k|^2\,d\sigma
+O(s_k)\rightarrow
\frac1n\int_{\mathbb S^n}|u|^2\log |u|^2\,d\sigma.
\end{aligned}
\end{equation}
We next pass to the lower limit in the quadratic part.
The energy bound,  Lemma \ref{lem:spectral-domination} and \eqref{eq:entropy-moving} give
\[
0\le
 \sum_{\ell,m}\Lambda_\ell|(u_k)_{\ell,m}|^2
 \le \frac{Q_k(u_k)-\|u_k\|_2^2}{2s_k}
\le \frac K2+\frac{1-\|u_k\|_2^2}{2s_k}
=O(1)
\]
for all sufficiently large $k$.
Strong $L^2$ convergence implies convergence of every
spherical harmonic coefficient. Fatou lemma   gives
\[
\sum_{\ell,m}\Lambda_\ell|u_{\ell,m}|^2\le
\liminf_{k\to\infty} \sum_{\ell,m}\Lambda_\ell|(u_k)_{\ell,m}|^2
\le
\liminf_{k\to\infty} \frac{Q_k(u_k)-\|u_k\|_2^2}{2s_k}
<\infty.
\]
Thus $u\in\mathcal X$.
\medskip

Finally, the normalization implies
$\mathcal D_{s_k} (u_k)
=Q_k(u_k)-1$.
Combining the preceding estimates and using $\|u\|_2=1$,
we conclude that
\[
\liminf_{k\to\infty} \frac{1}{2s_{k}}\mathcal D_{s_k} (u_k)
\ge
\langle u,\mathcal H u\rangle
-\frac1n\int_{\mathbb S^n}|u|^2\log|u|^2\,d\sigma
=
\mathcal D(u).
\]
\end{proof}

The remaining case is convergence to $\cM$, where the limiting denominator
vanishes. The next proposition combines the uniform entropy expansion with
\eqref{eq:residual-bound} to recover the local one-bubble threshold.

\begin{proposition}
\label{prop:one-bubble}
Let $s_k\rightarrow 0^+$,  and  let $u_k$ be normalized
minimizers for $c_{s_k}$
satisfying $\sup\limits_k\|u_k\|_\infty< \infty$.
  If $u_k\rightarrow h\in\cM_{0,1}$  strongly in $L^2(\Sn)$,
then
\[ \liminf_{k\to\infty}c_{s_k}
 \geq\frac{2}{n+2}.
\]
\end{proposition}

\begin{proof}
Pass to a subsequence on which $c_{s_k}$ converges to its lower limit. Choose
active normalized bubbles $\widehat h_k$ with positive correlation. Since
$u_k\to h\in\cM_{0,1}$ in $L^2$, the uniform $L^\infty$ bound implies
$|u_k|^{p_k}\to h^2$ in $L^1$. After extraction, $\widehat h_k\to h$ in $L^2$. Applying the
corresponding conformal action and a sign change, we  assume
\[
 u_k\to1\quad\text{in }L^2,
 \qquad 1\in\mathcal A_{s_k}(u_k).
\]
Set
\[
 t_k=\int_{\Sn}u_k\,d\sigma,
 \qquad z_k=u_k-t_k,
 \qquad \varepsilon_k=\|z_k\|_2,
 \qquad N_k=\|u_k\|_2^2=t_k^2+\varepsilon_k^2.
\]
Then $t_k\to1$, $\varepsilon_k\to0$, and $\varepsilon_k>0$. Since $1$ is
active,
\begin{equation}\label{eq:local-distance}
 m_k(u_k)=t_k^2,
 \qquad Q_k(u_k)=t_k^2+Q_k(z_k),
 \qquad \dist_{s_k}(u_k,\cM_{s_k})^2=Q_k(z_k).
\end{equation}
Differentiating the active correlation at $\xi=0$ gives
$z_k\perp\mathcal H_0\oplus\mathcal H_1$.
 \eqref{eq:residual-bound} tested against $z_k$, gives
\[
 \left|Q_k(z_k)-\int_{\Sn}|u_k|^{p_k-2}u_kz_k\,d\sigma\right|
 \le2s_kc_{s_k}Q_k(z_k).
\]
Because $\|u_k\|_\infty$ is uniformly bounded, the map
$x\mapsto|x|^{p_k-2}x$ is $(1+O(s_k))$-Lipschitz on the relevant range.
Using $\int_{\Sn} z_k \,d\sigma=0$   yields
\[
 \varepsilon_k^2\le Q_k(z_k)\le(1+Cs_k)\varepsilon_k^2.
\]

Set
$w_k:=\frac{z_k}{\varepsilon_k}$,
\[
 \|w_k\|_2=1,
 \qquad w_k\perp\mathcal H_0\oplus\mathcal H_1,
 \qquad Q_k(w_k)\to1,
\]
and spectral domination gives
\[
 \Lambda_2\le\langle w_k,\cH w_k\rangle
 \le\frac{Q_k(w_k)-1}{2s_k}\le C.
\]
After extraction, $w_k\to w$ strongly in $L^2$.

Apply Lemma \ref{lem:uniform-entropy-expansion} to
$F_k=\frac{u_k^2}{N_k}$. Since
$\int_{\Sn} F_k^{\frac{p_k}{2}}\,d\sigma=N_k^{-\frac{p_k}{2}}$,
\begin{equation}\label{eq:N-second}
 \frac{1-N_k}{2s_k\varepsilon_k^2}\longrightarrow\frac2n.
\end{equation}
Using \eqref{eq:local-distance},
\[
 Q_k(u_k)=N_k+\varepsilon_k^2(Q_k(w_k)-1),
 \qquad
 \dist_{s_k}(u_k,\cM_{s_k})^2=\varepsilon_k^2Q_k(w_k).
\]
Hence
\[
 c_{s_k}
 =\frac{\dfrac{Q_k(w_k)-1}{2s_k}
       -\dfrac{1-N_k}{2s_k\varepsilon_k^2}}
       {Q_k(w_k)}
 \ge
 \frac{\Lambda_2-\dfrac{1-N_k}{2s_k\varepsilon_k^2}}
      {Q_k(w_k)}.
\]
Letting $k\to\infty$ and using \eqref{eq:N-second} gives
\[
 \liminf_{k\to\infty}c_{s_k}
 \ge\Lambda_2-\frac2n=\frac{2}{n+2}.
\]
\end{proof}

Together with the strict local gap \eqref{eq:strict-local-intro},
Proposition \ref{prop:one-bubble} excludes convergence to the optimizer manifold
whenever $\liminf\limits_{k\to\infty}c_{s_k}\le c_{\rm LS}$.

\begin{proof}[Proof of Theorem \ref{thm:main}]
Let $s_k\to0^+$   and pass to a subsequence such that
$c_{s_k}\to\liminf\limits_{s\to0^+}c_s$. Choose normalized minimizers $u_k$ and
apply Hersch normalization, by Corollary \ref{cor:sharp-recovery},
\begin{equation}\label{eq:liminf-upper}
 \liminf_{s\to0^+}c_s\le c_{\rm LS}.
\end{equation}
For all sufficiently large $k$, the strict moving gap
\eqref{eq:strict-moving-gap-intro} holds and $Q_k(u_k)\le1+Cs_k$.
Sections \ref{sec:localization}-\ref{sec:uniform-boundedness} and
Proposition \ref{prop:strong-L2} therefore give, after extraction,
\[
 u_k\longrightarrow u\quad\text{strongly in }L^2(\Sn),
 \qquad \|u\|_2=1,
\]
together with \eqref{eq:distance-limit}, Proposition
\ref{prop:numerator-liminf} gives $u\in\cX$.

If $u\in\cM$, then $u\in\cM_{0,1}$ and Proposition
\ref{prop:one-bubble} yields
$\liminf\limits_{s\to0^+}c_s\ge\frac2{n+2}$, contradicting
\eqref{eq:liminf-upper} and \eqref{eq:strict-local-intro}. Hence
$\dist(u,\cM)>0$, and the lower-limit inequality gives
\[
 \liminf_{s\to0^+}c_s
 \ge\frac{\cD(u)}{\dist(u,\cM)^2}
 \ge c_{\rm LS}.
\]
Together with \eqref{eq:liminf-upper}, this proves
\[
 \liminf_{s\to0^+}c_s=c_{\rm LS},
 \qquad
 c_{\rm LS}=\frac{\cD(u)}{\dist(u,\cM)^2},
\]
so $u$ is a minimizer. Since the original sequence $s_k\to0^+$ was
arbitrary, Corollary \ref{cor:sharp-recovery} also gives
$c_s\to c_{\rm LS}$.

Finally, let $s_k\to0^+$ and let $u_k$ be normalized fractional
minimizers for $c_{s_k}$, the same compactness argument yields a
subsequence $u_{k_j}$ that converges strongly in $L^2$.
The one-bubble bound shows that the limit does not belong to $\cM$.
The lower-limit estimate then shows that the limit is a normalized
minimizer for $c_{\rm LS}$.
 This proves  Theorem \ref{thm:main}.
\end{proof}

\appendix
% elsarticle includes "Appendix" in section references.  Theorem numbers
% retain the conventional form A.1 rather than "Appendix A.1".
\renewcommand{\thetheorem}{\Alph{section}.\arabic{theorem}}

\section{Approximation and recovery}
\label{app:approximation}
\renewcommand{\theequation}{\Alph{section}.\arabic{equation}}
We collect the approximation results used in Sections \ref{sec:prelim} and \ref{sec:localization}.

\begin{proof}[Proof of Lemma \ref{lem:smooth-recovery-density}]
If $u=0$, take $v_j=0$. Hence assume $u\ne0$.
Let $S_t=e^{t\Delta_{\mathbb S^n}}$ and set $v_t=S_tu$.
Then $v_t\in C^\infty(\mathbb S^n)$ for $t>0$.
As $t\rightarrow 0^+$,
the spherical harmonic expansion gives
\[
 \|v_t-u\|_{\mathcal{X}}^2
 =
 \sum_{\ell,m}
 (1+\Lambda_\ell)
 |e^{-t\ell(\ell+n-1)}-1|^2|u_{\ell,m}|^2 \rightarrow 0
\]
by dominated convergence.
 Define
\[
\Phi(x)=
\begin{cases}
x^{2} \log x^{2}, & x \neq 0,\\
0, & x=0,
\end{cases}
\]
 and let
$\Phi_\pm=\max\{\pm\Phi,0\}$.
By Theorem A and the bound $\Phi_-\le e^{-1}$,
we have $\Phi(u)\in L^1(\mathbb S^n)$.
The function $\Phi_+$ is convex, so Jensen inequality
for the heat semigroup and preservation of the integral give
\[
 \int_{\mathbb S^n}\Phi_+(v_t)\,d\sigma
 \le
 \int_{\mathbb S^n}S_t\bigl(\Phi_+(u)\bigr)\,d\sigma
 =
 \int_{\mathbb S^n}\Phi_+(u)\,d\sigma.
\]
Since $v_t\to u$ in measure, Fatou lemma yields
\[
 \int_{\mathbb S^n}\Phi_+(u)\,d\sigma
 \le
 \liminf_{t\rightarrow 0^+}
 \int_{\mathbb S^n}\Phi_+(v_t)\,d\sigma.
\]
Thus the positive-part integrals converge.
The function $\Phi_-$ is continuous and bounded, so
$\Phi_-(v_t)\to\Phi_-(u)$ in $L^1(\mathbb S^n)$.
Consequently,
\[
 \int_{\mathbb S^n}\Phi(v_t)\,d\sigma
 \longrightarrow
 \int_{\mathbb S^n}\Phi(u)\,d\sigma.
\]
Using $\|v_t\|_2\to\|u\|_2$ and
\[
 \int_{\mathbb S^n}
 f^2\log\frac{f^2}{\|f\|_2^2}\,d\sigma
 =
 \int_{\mathbb S^n}\Phi(f)\,d\sigma
 -\|f\|_2^2\log\|f\|_2^2,
\]
we obtain convergence of the normalized entropy.
Taking any sequence $t_j\rightarrow0^+$ and setting
$v_j=S_{t_j}u$ proves the form and entropy convergence,
and hence $\mathcal{D}(v_j)\to \mathcal{D}(u)$.
\medskip

It remains to compare the distances and the quotients.
The distance function is $1$-Lipschitz,
\[
 \bigl|
 \operatorname{dist}(v_j,\cM)
 -
 \operatorname{dist}(u,\cM)
 \bigr|
 \le \|v_j-u\|_2\longrightarrow0.
\]
If $u\notin\cM$, then
$\operatorname{dist}(u,\cM)>0$.
Indeed, otherwise Lemma \ref{lem:distance-formula} would give an attained
correlation $m(u)=\|u\|_2^2>0$, and equality in
Cauchy-Schwarz inequality
implies $u\in\cM$.
After discarding finitely many terms, therefore,
$v_j\notin\cM$ and
\[
 \frac{\mathcal{D}(v_j)}{\operatorname{dist}(v_j,\cM)^2}
 \longrightarrow
 \frac{\mathcal D(u)}{\operatorname{dist}(u,\cM)^2}.
\]
Taking infima, together with the inclusion
$C^\infty(\mathbb S^n)\setminus\cM
\subset \mathcal{X} \setminus\cM$, proves
\eqref{eq:smooth-quotient-density}.
\end{proof}
\medskip

We next prove a fixed-order density result.
It extends the localization estimates from compactly supported cylinder representatives
to general functions in the form domain.

\begin{lemma}
\label{lem:cylinder-compact-core}
For fixed $0<s<\min\{1,\frac{n}{2}\}$, smooth functions on $\mathbb S^n$
vanishing near the two prescribed poles are dense in the form norm
\[
\bigl(Q_s(w)+\|w\|_2^2\bigr)^{\frac{1}{2}}.
\]
Equivalently, their conformal-cylinder representatives form a core
for $Q_s$.
\end{lemma}

\begin{proof}
The Gamma-ratio asymptotics in \eqref{eq:frac-eigen} give,
for fixed $s$,
\[ c_{n,s}(1+\ell)^{2s}
 \le 1+r_{s,\ell}
 \le C_{n,s}(1+\ell)^{2s}.\]
Since the eigenvalues of $-\Delta_{\mathbb S^n}$ are
$\ell(\ell+n-1)$, this form norm
 is equivalent to the usual
$H^s(\mathbb S^n)$ norm. Finite sums of spherical harmonics
are therefore dense in the form domain, so it suffices to
approximate a fixed $f\in C^\infty(\mathbb S^n)$.
\medskip

Let $P_\pm$ be the prescribed poles. For sufficiently small
$\varepsilon>0$, choose smooth cutoffs
$\rho_{\varepsilon,\pm}$ with disjoint supports such that
\[
 0\le\rho_{\varepsilon,\pm}\le1,
 \qquad
 \rho_{\varepsilon,\pm}=1
 \quad\text{on }B_\varepsilon(P_\pm),
\]
\[
 \operatorname{supp}\rho_{\varepsilon,\pm}
 \subset B_{2\varepsilon}(P_\pm),
 \qquad
 \|\nabla\rho_{\varepsilon,\pm}\|_\infty
 \le C\varepsilon^{-1}.
\]
Set $g_{\varepsilon,\pm}=f\rho_{\varepsilon,\pm}$.
The support and gradient bounds imply
\[
 \|g_{\varepsilon,\pm}\|_2^2
 \le C_f\varepsilon^n,
 \qquad
 \|g_{\varepsilon,\pm}\|_{H^1}^2
 \le C_f\varepsilon^{n-2}.
\]
Since $0<s<1$, interpolation in the spherical harmonic
expansion gives
\[
 \|g_{\varepsilon,\pm}\|_{H^s}
 \le
 \|g_{\varepsilon,\pm}\|_2^{1-s}
 \|g_{\varepsilon,\pm}\|_{H^1}^{s}
 \le C_f\varepsilon^{\frac{n}{2}-s}
 \longrightarrow0.
\]

Define
$ f_\varepsilon
 :=f(1-\rho_{\varepsilon,+})(1-\rho_{\varepsilon,-})$.
The cutoff supports are disjoint, so
$f-f_\varepsilon=g_{\varepsilon,+}+g_{\varepsilon,-}$.
Using the equivalence of the form norm and the $H^s$ norm, we obtain
\[
\begin{aligned}
\bigl(
 Q_s(f-f_\varepsilon)+\|f-f_\varepsilon\|_2^2
\bigr)^{\frac12}
&\le
C_{n,s}\|f-f_\varepsilon\|_{H^s}\le
C_{f,n,s}\varepsilon^{\frac n2-s}
\longrightarrow0.
\end{aligned}
\]
This proves the density claim.
 By the cylinder transform, these are
precisely the smooth functions with compactly supported
cylinder representatives.
\end{proof}

\section{Discrete mixing inequalities and local entropy expansion}
\label{app:entropy}

We record the mixing inequality and the local entropy expansion used in the compactness argument.

\begin{lemma}\label{lem:fractional-mixing}
Let $0<\alpha<1$, $\{a_j\}$ be a finite or countable probability vector, and set
$ S_\alpha:=\sum\limits_ja_j^\alpha$,  $A:=\sup\limits_ja_j$.
Then
\begin{equation}\label{eq:fractional-mixing}
 S_\alpha-1\geq(2-2^\alpha)(S_\alpha-A^\alpha).
\end{equation}
If $0\leq q_j\leq a_j^\alpha$, then
\begin{equation}\label{eq:fractional-mixing-q}
 S_\alpha-1\geq(2-2^\alpha)\left(\sum_jq_j-\sup_jq_j\right).
\end{equation}
\end{lemma}
\begin{proof}
Both inequalities are immediate in the extended sense if
$S_\alpha=+\infty$. Hence assume $S_\alpha<\infty$.
Since $\{a_j\}$ is a probability vector, its largest mass
is attained, choose $j_0$ with $a_{j_0}=A$.
Let $d_\alpha=2^\alpha-1\in(0,1)$,
 \eqref{eq:fractional-mixing} is equivalent to
\[
 d_\alpha S_\alpha+(1-d_\alpha)A^\alpha\ge1.
\]
Suppose first that $A\ge\frac12$.
Subadditivity of $x^\alpha$ gives
\[
 S_\alpha-A^\alpha
 =
 \sum_{j\ne j_0}a_j^\alpha
 \ge (1-A)^\alpha.
\]
Since the function
$x^\alpha+d_\alpha(1-x)^\alpha$ is strictly concave
on $[\frac12,1]$ and the two endpoint values,
its values  are at least $1$ throughout the interval.
Then we deduce that
\[
 d_\alpha S_\alpha+(1-d_\alpha)A^\alpha
 \ge A^\alpha+d_\alpha(1-A)^\alpha
 \ge1.
\]

Now suppose that $A<\frac12$.
Since $a_j^\alpha\ge A^{\alpha-1}a_j$, we have
$S_\alpha\ge A^{\alpha-1}$.
Set $x=A^{-1}\ge2$,
then
\[
 x^\alpha
 \le 1+(x-1)(2^\alpha-1)
 =1+d_\alpha(x-1),
\]
with strict inequality for $x>2$. Hence
\[
 d_\alpha S_\alpha+(1-d_\alpha)A^\alpha
 \ge
 A^\alpha\bigl[1+d_\alpha(x-1)\bigr]
 \ge1.
\]
Finally, if $0\le q_j\le a_j^\alpha$, then
\[
 \sum_jq_j-\sup_jq_j
 \le
 \sum_{j\ne j_0}q_j
 \le
 \sum_{j\ne j_0}a_j^\alpha
 =
 S_\alpha-A^\alpha.
\]
Combining this with \eqref{eq:fractional-mixing} proves \eqref{eq:fractional-mixing-q}.
\end{proof}

\begin{remark}
For a finite probability vector, dividing
\eqref{eq:fractional-mixing} by $1-\alpha$ and letting
$\alpha\to1^-$ gives
\[
-\sum_j a_j\log a_j\ge 2\log2\,(1-A).
\]
For a countable probability vector, set
\[
b_N:=\sum_{j>N}a_j,
\qquad
A_N:=\max\{a_1,\ldots,a_N,b_N\}.
\]
Applying the finite-dimensional inequality to
$(a_1,\ldots,a_N,b_N)$ yields
\[
-\sum_{j=1}^N a_j\log a_j-b_N\log b_N
\ge 2\log2\,(1-A_N).
\]
Since $b_N\to0$, $-b_N\log b_N\to0$, and
$A_N\to A:=\sup_j a_j$, while the partial entropy sums increase to
the full entropy, letting $N\to\infty$ gives
\[
-\sum_j a_j\log a_j\ge 2\log2\,(1-A).
\]
Here $0\log0:=0$.
\end{remark}

We now prove the expansion at the perturbation scale, without a rate condition
between $\delta_k$ and $\varepsilon_k$.

\begin{lemma}
\label{lem:uniform-entropy-expansion}
Let $t_k\to1$, $\eps_k\rightarrow 0^+$, and   $w_k\to w$ strongly in
$L^2(\Sn)$, with $\int_{\Sn} w_k\dd\sigma=0$ and $\norm{w_k}_2=1$.  Assume
that
\[
 u_k=t_k+\eps_kw_k,
 \qquad \sup_k\norm{u_k}_\infty<\infty,
 \qquad N_k=\norm{u_k}_2^2=t_k^2+\eps_k^2.
\]
If $\delta_k\rightarrow 0^+$, $F_k:=
\frac{u_k^2}{N_k}$,
then
\begin{equation}\label{eq:Renyi-expansion-lemma}
\frac1{\delta_k}\log\int_{\Sn}F_k^{1+\delta_k}\dd\sigma
=2\eps_k^2+o(\eps_k^2).
\end{equation}
\end{lemma}
\begin{proof}
Set $z_k:=\varepsilon_kw_k$,
then $N_k\to1$, $\int_{\mathbb S^n} F_k\,d\sigma=1$, and $\{F_k\}$ is
uniformly bounded.
Let $\Phi(x)=x^2\log x^2$, with $\Phi(0)=0$, and define
\[
\mathcal R_k
:=
\Phi(t_k+z_k)-\Phi(t_k)-\Phi'(t_k)z_k
-\frac12\Phi''(t_k)z_k^2.
\]
For each fixed $0<\eta<\frac 14$, Taylor expansion  away from zero
and the uniform boundedness of $u_k$ give
\[
|\mathcal R_k|
\le
\begin{cases}
\omega(\eta)z_k^2, & |z_k|\le\eta,\\
C_\eta z_k^2, & |z_k|>\eta,
\end{cases}
\qquad
\omega(\eta)\to0.
\]
Since $w_k\to w$ strongly in $L^2$, the sequence $\{w_k^2\}$
is uniformly integrable. Hence
\[
\int_{\{|z_k|>\eta\}}z_k^2\,d\sigma
=
\varepsilon_k^2
\int_{\{|w_k|>\eta/\varepsilon_k\}}w_k^2\,d\sigma
=
o(\varepsilon_k^2).
\]
Letting first $k\to\infty$ and then $\eta\rightarrow 0^+$ yields
$\int_{\Sn}|\mathcal R_k|\,d\sigma=o(\varepsilon_k^2)$.
Using $\int_{\Sn} z_k\,d\sigma=0$,
$\int_{\Sn} z_k^2\,d\sigma=\varepsilon_k^2$, and
$\frac{1}{2}\Phi''(t)=\log t^2+3$, we obtain
\[
\int_{\Sn}\Phi(u_k)\,d\sigma
=
t_k^2\log t_k^2
+(\log t_k^2+3)\varepsilon_k^2
+o(\varepsilon_k^2).
\]
Also,
\[
N_k\log N_k
=
t_k^2\log t_k^2
+(\log t_k^2+1)\varepsilon_k^2
+O(\varepsilon_k^4).
\]
Subtracting gives
\[
\int_{\Sn} u_k^2\log\frac{u_k^2}{N_k}\,d\sigma
=
2\varepsilon_k^2+o(\varepsilon_k^2),
\]
therefore
\[
\int_{\Sn} F_k\log F_k\,d\sigma
=
2\varepsilon_k^2+o(\varepsilon_k^2).
\]

For bounded $y\ge0$, one has
$y(\log y)^2\le C(\sqrt y-1)^2$,
because the ratio extends continuously at $y=0$ and $y=1$.
Consequently,
\[
\begin{aligned}
\int_{\Sn} F_k(\log F_k)^2\,d\sigma&\le
\frac C{N_k}
\int_{\Sn}\bigl(|u_k|-\sqrt{N_k}\bigr)^2\,d\sigma\le
\frac C{N_k}
\left[
\varepsilon_k^2+(t_k-\sqrt{N_k})^2
\right]
=
O(\varepsilon_k^2),
\end{aligned}
\]
where
$\sqrt{N_k}-t_k
=\frac{\varepsilon_k^2}{\sqrt{N_k}+t_k}=O(\varepsilon_k^2)$.

Finally, for bounded $y\ge0$ and $0<\delta<\frac{1}{2}$,
Taylor formula with integral remainder gives
\[
0\le y^{1+\delta}-y-\delta y\log y
\le C\delta^2y(\log y)^2.
\]
Thus
\[
\int_{\Sn} F_k^{1+\delta_k}\,d\sigma
\le1+\delta_k  \int_{\Sn} F_k\log F_k\,d\sigma+C\delta_k^2 \int_{\Sn} F_k(\log F_k)^2\,d\sigma.
\]
Jensen inequality and $\log x\le x-1$,  yield
\[
\int_{\Sn} F_k\log F_k\,d\sigma
\le
\frac1{\delta_k}\log\int_{\Sn}F_k^{1+\delta_k}
\le
\int_{\Sn} F_k\log F_k\,d\sigma+C\delta_k \int_{\Sn} F_k(\log F_k)^2\,d\sigma.
\]
Since $\delta_k\to0$ and $\int_{\Sn} F_k(\log F_k)^2\,d\sigma=O(\varepsilon_k^2)$,
\[
\frac1{\delta_k}\log\int_{\Sn}F_k^{1+\delta_k}\dd\sigma
=2\varepsilon_k^2+o(\varepsilon_k^2).
\]
This proves the lemma.
\end{proof}
\section{Localization and the concentration-compactness alternative}
\label{app:scale-localization}
We prove the localization alternative in Proposition \ref{prop:scale-cc}. Fix its conformal cylinder and write
\[ \omega(t,\theta)=(\operatorname{sech}t\,\theta,\tanh t),
 \qquad(t,\theta)\in\R^1 \times\mathbb S^{n-1}.
\]
  \[ \dd\mu=\abs{\mathbb S^n}^{-1}\dd t\dd S_\theta, \qquad
\dd\sigma=\operatorname{sech}^nt\dd\mu, \qquad
v(t,\theta)=\operatorname{sech}^{\frac{n-2s}{2}}t\,u(\omega(t,\theta)).
\]
Then $\norm{v}_{L^{p_s}(\dd\mu)}=\norm{u}_{p_s}$. Since $\dd\mu$ has infinite mass,   $L^2(\dd\mu)$ estimates are first proved for compactly supported representatives and extended by Lemma \ref{lem:cylinder-compact-core}.

\begin{lemma}\label{lem:slow-IMS}
Let $\vartheta\in C^1(\mathbb{R}^1,[0,\frac{\pi}{2}])$ be zero on $(-\infty,-1]$, equal to $\frac{\pi}{2}$ on $[1,\infty)$, and have bounded derivative. For $L\geq2$, $T\in \mathbb{R}^1$, set
\[
 \chi_1=\cos\vartheta(\frac{t-T}{L}),
 \qquad \chi_2=\sin\vartheta(\frac{t-T}{L}).
\]
Then, for every   $u$ in $\mathcal X_s$,
\begin{equation}\label{eq:slow-IMS}
 \abs{Q_s(\chi_1u)+Q_s(\chi_2u)-Q_s(u)}
 \leq C_{n,s_0,\vartheta}sL^{-2+\frac{2s}{n}}\norm{u}_{p_s}^2,
\end{equation}
with a constant independent of $T$.
\end{lemma}

\begin{proof}
Write $\mathcal C=\mathbb R^1\times\mathbb S^{n-1}$ and
$\chi(t)=\big(\chi_1(t),\chi_2(t)\big)$. For fixed $T,L$, the cutoffs are constant near the poles and
extend to Lipschitz functions on $\mathbb S^n$. Moreover,
\[
\Gamma_{T,L}(t,t')
:=
|\chi(t)-\chi(t')|^2
\le
C_\vartheta
\min\left\{
1,\frac{|t-t'|^2}{L^2}
\right\}.
\]
And $\Gamma_{T,L}(t,t')=0$ if $t,t'\le T-L$ or if
$t,t'\ge T+L$.
By Lemma \ref{lem:IMS},
\[
 \abs{\omega(t,\theta)-\omega(t',\theta')}^2
 =\frac{2(\cosh(t-t')-\theta\cdot\theta')}{\cosh t\cosh t'},
\]
 and the definition of $v$, we obtain
\[
\begin{aligned}
&\left|
Q_s(\chi_1u)+Q_s(\chi_2u)-Q_s(u)
\right|\\
&\quad\le
B_{n,s}
\int_{\mathcal C} \int_{\mathcal C}
|v(t,\theta)v(t',\theta')|
K_s(t-t',\theta\cdot\theta')
\Gamma_{T,L}(t,t')
\,d\mu\,d\mu',
\end{aligned}
\]
where
\[
K_s(\tau,z)=
[2(\cosh\tau-z)]^{-\frac{n+2s}{2}}.
\]
Using
$2(\cosh\tau-\cos\phi)\geq C (\tau^2+\phi^2)$  for $|\tau|\le1$ and $0\le\phi\le\pi$,
 one obtains, uniformly
for $0<s<s_0$,
\[
\int_{\mathbb S^{n-1}}
K_s(\tau,\theta\cdot\theta')\,dS_{\theta'}
\le
C_{n,s_0}
\begin{cases}
|\tau|^{-1-2s},&0<|\tau|\le1,\\
e^{-\frac{n|\tau|}{2}},&|\tau|\ge1.
\end{cases}
\]
\medskip

Set
$d_T(t):=\operatorname{dist}(t,[T-L,T+L])$.
The support property of $\Gamma_{T,L}$ and the preceding kernel bound
imply
\[
\int_{\mathcal C}
K_s\Gamma_{T,L}\,d\mu'
\le
\frac{C_{n,s_0,\vartheta}}{L^2} \int_{|\tau|\ge d_T(t)}
\int_{\mathbb S^{n-1}} \tau^2
K_s(\tau,\theta\cdot\theta')\,dS_{\theta'}
\,d\tau
\le
\frac{C_{n,s_0,\vartheta}}{L^2} e^{-c_nd_T(t)}.
\]
Here $s_0<1$ ensures uniform integrability at the origin,
and the exponential tail absorbs the factor $\tau^2$.
By $2|vv'|\le |v|^2+|v'|^2$ and symmetry, the double integral
 is  bounded by
\[\frac{C_{n,s_0,\vartheta}}{L^2}
\int_{\mathcal C} |v(t,\theta)|^2e^{-c_nd_T(t)}\,d\mu.\]
H\"older inequality with exponents  $\frac{p_s}{2}$
 and $\frac{n}{2s}$  gives,
\[
\begin{aligned}
\int_{\mathcal C} |v|^2e^{-c_nd_T}\,d\mu
&\le
\|v\|_{p_s}^2
\left(
\int_{\mathcal C} e^{-c_nn \frac{d_T}{2s} }\,d\mu
\right)^{\frac{2s}{n}}\le
C_{n,s_0}L^{\frac{2s}{n}}\|u\|_{p_s}^2.
\end{aligned}
\]
Multiplying by $B_{n,s}\leq C_{n,s_0}s$ proves
\eqref{eq:slow-IMS}.
All constants are independent of $T$.
\end{proof}

The next lemma excludes vanishing in the cylinder variable under the bound $Q_s(u)
\leq 1 + Ks$. Choose a smooth locally finite partition
$\{\chi_j^{(L)}\}_{j\in\mathbb Z}$ depending only on $t$ such that
\begin{equation}\label{eq:periodic-partition}
 \sum_j(\chi_j^{(L)})^2=1,
 \quad
 \supp\chi_j^{(L)}\subset\{\abs{t-jL}\leq C_0L\},
 \quad
 \sum_j\abs{\partial_t\chi_j^{(L)}}^2\leq C_0L^{-2},
\end{equation}
with uniformly bounded overlap.

\begin{lemma}\label{lem:periodic-window}
Let $u$ belong to $\mathcal X_s$ and satisfy
$\norm{u}_{p_s}=1$. For its cylinder representative $v$, set
$ b_j=\norm{\chi_j^{(L)}v}_{p_s}^{p_s}$.
For every $K\ge 0$ and every fixed $L\geq2$, there are $\delta>0$ and $s_1>0$ such that, if $Q_s(u)\leq1+Ks$ and $0<s<s_1$, then
\begin{equation}\label{eq:periodic-window-conclusion}
 \sup_{T\in\mathbb R^1}\int_{\abs{t-T}\leq C_0L}\abs v^{p_s}\dd\mu\geq\delta.
\end{equation}
\end{lemma}

\begin{proof}
Write
$\chi_j:=\chi_j^{(L)}$
and let $N_*$ be a uniform overlap bound.
All norms of cylinder functions below are taken with respect to
$d\mu$. Constants depend only on $n,s_0$ and the partition bounds,
unless otherwise indicated.

First suppose that $v$ is smooth and compactly supported.
By \eqref{eq:periodic-partition},
\[
\sum_j|\chi_j(t)-\chi_j(t')|^2
\le
C\min\left\{1,\frac{|t-t'|^2}{L^2}\right\}.
\]
The uniform kernel estimate in the proof of
Lemma \ref{lem:slow-IMS}, together with
Lemma \ref{lem:IMS} and $B_{n,s}\le Cs$, gives
\[
\left|
\sum_jQ_s(\chi_ju)-Q_s(u)
\right|
\le
CsL^{-2}\|v\|_{L^2(d\mu)}^2.
\]
Since each cutoff support has $d\mu$-measure at most $CL$,
H\"older inequality and the quadratic partition yield
\[
\|v\|_{L^2(d\mu)}^2
=
\sum_j\|\chi_jv\|_{L^2(d\mu)}^2
\le
CL^{\frac{2s}{n}}\sum_jb_j^{1-\frac{2s}{n}}.
\]
Combining this with the sharp Sobolev inequality
$Q_s(\chi_ju)\ge b_j^{1-\frac{2s}{n}}$, we obtain
\[
Q_s(u)
\ge
\left(1-CsL^{-2+\frac{2s}{n}}\right)\sum_jb_j^{1-\frac{2s}{n}}.
\]
If necessary, decrease $s_0$ so that
\[
1-CsL^{-2+\frac{2s}{n}}\geq \frac12
\]
for $0<s<s_0$ and $L\ge2$.
Fix $s$ and $L$, let $u^{(m)}$ be cylinder-compact
approximations of a general form-domain function $u$.
Then
\[
u^{(m)}\to u\quad\text{in }L^{p_s},
\qquad
Q_s(u^{(m)})\to Q_s(u).
\]
For each fixed $j$,
\[
b_j^{(m)}\to b_j.
\]
The preceding estimate does not require normalization.
Fatou lemma therefore gives
\[
Q_s(u)\geq
\left(1-CsL^{-2+\frac{2s}{n}}\right)
\sum_j b_j^{1-\frac{2s}{n}}.
\]
\medskip

Next, finite-dimensional H\"older inequality
and
$\sum\limits_j|\chi_j|^2=1$
give the pointwise bounds
$N_*^{1-\frac{p_s}{2}}
\le
\sum\limits_j|\chi_j|^{p_s}
\le1$.
Integrating against $|v|^{p_s}$, and using
$p_s-2=O(s)$, we obtain
\[
1-Cs
\le
N_*^{1-\frac{p_s}{2}}
\le
\sum\limits_jb_j
\le1.
\]
We now use  the lower-energy bound to control the largest
cell mass. Set $B:=\sup\limits_jb_j$. Since
 $Q_s(u)\le1+Ks$,
\[
\sum_jb_j^{1-\frac{2s}{n}}
\ge
B^{-\frac{2s}{n}}\sum_jb_j
\ge
B^{-\frac{2s}{n}}N_*^{1-\frac{p_s}{2}}.
\]
Consequently,
\[
1+Ks
\ge
\left(1-CsL^{-2+\frac{2s}{n}}\right)
B^{-\frac{2s}{n}}N_*^{1-\frac{p_s}{2}}.
\]
For all sufficiently small $s$, taking logarithms gives
\[
\begin{aligned}
\frac{2s}{n}\log\frac1B
&\le
\log(1+Ks)
-\log\left(1-CsL^{-2+\frac{2s}{n}}\right)
+\left(\frac {p_s}{2}-1\right)\log N_*\le C_{n,s_0,L}(K+1)s.
\end{aligned}
\]
Thus $B\ge\delta>0$, where $\delta$ is independent of $s$ and $u$.
Since $|\chi_j|\le1$ and
$\operatorname{supp}\chi_j\subset\{|t-jL|\le C_0L\}$,
\[
\sup_{T\in\mathbb R}
\int_{|t-T|\le C_0L}|v|^{p_s}\,d\mu
\ge
\sup_jb_j
=
B
\ge\delta.
\]
This proves \eqref{eq:periodic-window-conclusion}.
\end{proof}

Next we will verify that a
single normalized bubble cannot correlate substantially with two pieces
separated by a diverging cylinder gap.
A cylinder bubble satisfies
\begin{equation}\label{eq:cylinder-bubble-density}
 \abs{h_{\xi,s}}^{p_s}\dd\sigma
 =\frac{(1-b^2)^{\frac{n}{2}}}
 {[\cosh(t-T)-be\cdot\theta]^n}\dd\mu
\end{equation}
for suitable $T\in\R ^1$, $b\in[0,1)$ and $e\in\mathbb S^{n-1}$. Hence
\begin{equation}\label{eq:bubble-cylinder-tail}
 \int_{\abs{t-T}\geq R}\abs{h_{\xi,s}}^{p_s}\dd\sigma
 \leq C_ne^{-c_nR},\qquad R\geq1,
\end{equation}
uniformly in $s$ and $\xi$.

\begin{lemma}\label{lem:corr-gap}
Let $A_k<B_k$, $B_k-A_k\to\infty$. Let $w_k^-$ and $w_k^+$ have bounded $L^{p_k}$ norms and cylinder supports in $\{t\leq A_k\}$ and $\{t\geq B_k\}$. Then
\begin{equation}\label{eq:corr-gap}
 m_k(w_k^-+w_k^+)
 =\max\{m_k(w_k^-),m_k(w_k^+)\}+o(1).
\end{equation}
The same statement holds for finitely many pieces with pairwise diverging gaps.
\end{lemma}

\begin{proof}
Write   $d_k:=B_k-A_k$.
For a normalized bubble $h$, let $T_h$ be its cylinder
center in \eqref{eq:cylinder-bubble-density}, and set
\[
 I_k^\pm(h)
 :=
 \int_{\mathbb S^n}w_k^\pm |h|^{p_k-2}h\,d\sigma.
\]
If $T_h\le \frac{A_k+B_k}{2}$, then  H\"older inequality and
\eqref{eq:bubble-cylinder-tail} give, for all sufficiently
large $k$,
\[
 \begin{aligned}
 |I_k^+(h)|
 &\le
 \|w_k^+\|_{p_k}
 \left(
 \int_{\{|t-T_h|\ge \frac{d_k}{2}\}}|h|^{p_k}\,d\sigma
 \right)^{\frac{1}{p_k'}} \le C e^{-c d_k}.
 \end{aligned}
\]
The constants are independent of $h$ and $k$, since
$\frac{1}{p_k'}=\frac{1}{2}+\frac{s_k}{n}\ge\frac{1}{2}$ and the $L^{p_k}$-norms are
uniformly bounded.
If $T_h>\frac{A_k+B_k}{2}$, the same estimate holds for
$I_k^-(h)$. Consequently,
\[
 \min\{|I_k^-(h)|,|I_k^+(h)|\}
 \le \varepsilon_k:=Ce^{-cd_k}\longrightarrow0
\]
uniformly over normalized bubbles.
The elementary inequality
\[\left||a+b|-\max\{|a|,|b|\}\right|
 \le \min\{|a|,|b|\},\]
 gives,
after taking suprema over normalized bubbles,
\[
 \left|
 m_k(w_k^-+w_k^+)^{\frac{1}{2}}
 -\max\{m_k(w_k^-)^{\frac{1}{2}},m_k(w_k^+)^{\frac{1}{2}}\}
 \right|
 \le\varepsilon_k.
\]
By H\"older inequality, $m_k(f)^{\frac{1}{2}}\le\|f\|_{p_k}$.
Hence all quantities above are uniformly bounded, and
squaring proves \eqref{eq:corr-gap}.
\medskip

For a fixed family of $J$ pieces, let $d_k\to\infty$ be
their minimum pairwise gap in the cylinder coordinate.
For every bubble center, at most one support meets
$\{|t-T_h|<\frac{d_k}{2}\}$.
Consequently, at most one of the corresponding pairings
$I_{j,k}(h)$ can be non-negligible. Every other pairing satisfies
\[
|I_{j,k}(h)|\leq Ce^{-cd_k},
\]
uniformly in $h$.
The triangle and reverse triangle inequalities yield
\[
\left||\sum_{j=1}^J I_{j,k}(h)|
-\max_{1\le j\le J}|I_{j,k}(h)|
\right|
\le C (J-1) e^{-cd_k}.
\]
Taking suprema and squaring as above proves the finite-piece
assertion.
\end{proof}

\begin{proof}[Proof of the alternatives in Proposition \ref{prop:scale-cc}]
Subsequences are not relabeled. Define
\[
 \nu_k(E):=\int_{E\times\mathbb S^{n-1}}|v_k|^{p_k}\,d\mu,
 \qquad
 q_k(R):=\sup_{T\in\mathbb R}\nu_k([T-R,T+R]).
\]
After a diagonal extraction, $q_k(R)\to q(R)$ for every positive
integer $R$. Set $q_\infty:=\lim_{R\to\infty}q(R)$.
Lemma \ref{lem:periodic-window} gives $q_\infty>0$.

If $q_\infty=1$, fix $R_0$ and centers $T_k$ such that
$\nu_k([T_k-R_0,T_k+R_0])>\frac34$ for all large $k$.
For each $0<\varepsilon<\frac14$, choose $R_\varepsilon$ so that,
for all large $k$, some interval of radius $R_\varepsilon$ carries
mass greater than $1-\varepsilon$. It intersects the first interval
and is contained in
$[T_k-R_0-2R_\varepsilon,T_k+R_0+2R_\varepsilon]$.
Thus the translated measures are tight, proving \textup{(S1)}.

Suppose $0<q_\infty<1$. A diagonal choice gives integer radii
$r_k\to\infty$ and $R_k>r_k$, with $R_k-r_k\to\infty$, such that
$q_k(r_k),q_k(R_k)\to q_\infty$.
Choose centers $T_k$ asymptotically maximizing $q_k(r_k)$. Then
\begin{equation}\label{eq:cc-inner-outer}
 \nu_k([T_k-r_k,T_k+r_k])=q_\infty+o(1),\qquad
 \nu_k([T_k-R_k,T_k+R_k])=q_\infty+o(1).
\end{equation}
Set
\[
 L_k:=\frac{R_k-r_k}{4},\qquad
 \tau_k^\pm:=T_k\pm\frac{R_k+r_k}{2},\qquad
 E_k^\pm:=[\tau_k^\pm-L_k,\tau_k^\pm+L_k].
\]
Then $L_k\to\infty$, and by \eqref{eq:cc-inner-outer},
$\nu_k(E_k)=o(1)$ for $E_k:=E_k^-\cup E_k^+$.
Using the function $\vartheta$ from Lemma \ref{lem:slow-IMS}, define
\[
 \phi_k(t):=\vartheta\!\left(\frac{t-\tau_k^-}{L_k}\right),\qquad
 \psi_k(t):=\vartheta\!\left(\frac{t-\tau_k^+}{L_k}\right),
\]
\[
 \chi_{-,k}:=\cos\phi_k,\qquad
 \chi_{0,k}:=\sin\phi_k\cos\psi_k,\qquad
 \chi_{+,k}:=\sin\phi_k\sin\psi_k.
\]
Set $u_k^j:=\chi_{j,k}u_k$.
Applying Lemma \ref{lem:slow-IMS} first to $u_k$ and then to
$\sin\phi_k\,u_k$, whose $L^{p_k}$ norm is at most $1$, gives
\begin{equation}\label{eq:cc-energy-full}
 \left|\sum_j Q_k(u_k^j)-Q_k(u_k)\right|
 \le Cs_kL_k^{-2+\frac{2s_k}{n}}=o(s_k).
\end{equation}
Since $\sum_j\chi_{j,k}^2=1$ and exactly one cutoff equals one
outside $E_k$, finite-dimensional H\"older inequality yields
\begin{equation}\label{eq:cc-mass-full}
 0\le1-\sum_j\|u_k^j\|_{p_k}^{p_k}
 \le C(p_k-2)\nu_k(E_k)=o(s_k).
\end{equation}
After a further extraction,
\[
 \|u_k^j\|_{p_k}^{p_k}\to a_j,\qquad
 a_0=q_\infty,\qquad a_-+a_+=1-q_\infty.
\]
At least two limiting masses are positive. We retain all three
pieces, including any with zero limiting mass.

Let $\widetilde u_k^j$ have cylinder representative
$\mathbf1_{E_k^c}\chi_{j,k}v_k$. Then
\[
 \left\|u_k-\sum_j\widetilde u_k^j\right\|_{p_k}
 +\sum_j\|u_k^j-\widetilde u_k^j\|_{p_k}
 \le4\nu_k(E_k)^{\frac1{p_k}}=o(1).
\]
The bulk supports are separated by gaps of width $2L_k\to\infty$.
These functions are used only in $L^{p_k}$ estimates, not in the
quadratic form. By H\"older inequality,
\[
 \left|m_k(f)^{\frac12}-m_k(g)^{\frac12}\right|
 \le\|f-g\|_{p_k},\qquad
 m_k(f)^{\frac12}\le\|f\|_{p_k}.
\]
These quantities are uniformly bounded, so squaring yields
\begin{equation}\label{eq:cc-corr-full}
 m_k(u_k)=\max_jm_k(u_k^j)+o(1).
\end{equation}
Together with \eqref{eq:cc-energy-full} and \eqref{eq:cc-mass-full},
this proves \textup{(S2)}. The alternatives are incompatible:
two bulk pieces with positive limiting masses and diverging
separation cannot both remain in a bounded interval after translation.
\end{proof}

\section{Inverse-kernel estimates and fixed-order regularity}
\label{app:regularity}
We prove the inverse-kernel bounds and fixed-order regularity used in Section \ref{sec:uniform-boundedness}.

\begin{proof}[Proof of Lemma \ref{lem:kernel-constants}]
Set
$c_n:=\frac{|\mathbb S^{n-1}|}{|\mathbb S^n|}$.
By the Funk-Hecke formula, the
integral operator in
\eqref{eq:inverse-kernel}
has multiplier
\[
\frac{\Gamma(\frac{n}{2}+s)}{\Gamma(\frac{n}{2}-s)}
\frac{\Gamma(\ell+\frac{n}{2}-s)}{\Gamma(\ell+\frac{n}{2}+s)}
\]
on spherical harmonics of degree $\ell$.
Thus the integral operator in
\eqref{eq:inverse-kernel} agrees with $P_s^{-1}$ on spherical
polynomials. By density of spherical polynomials in
$L^{p_s'}(\Sn)$ and the spherical HLS inequality, this identity
extends to every $g\in L^{p_s'}(\Sn)$ as an equality in
$L^{p_s}(\Sn)$. 
\medskip

Fix $\omega$ and write $\rho=|\omega-\eta|$.
Radial integration gives
\begin{equation}\label{eq:radial-prob}
\int_{\mathbb S^n}\varphi(\rho)\,d\sigma(\eta)
=
c_n\int_0^2
\varphi(\rho)\rho^{n-1}
\left(1-\frac{\rho^2}{4}\right)^{\frac{n-2}{2}}\,d\rho.
\end{equation}
Since $n\ge2$, the last factor is at most one.
Hence, for $0<R\le1$,
\[
\kappa_{n,s}
\int_{\rho<R}\rho^{2s-n}\,d\sigma
\le
\frac{\kappa_{n,s}c_n}{2s}R^{2s}
=
A_{n,s}R^{2s}.
\]
This proves \eqref{eq:kernel-near}.
The expression for $A_{n,s}$ in \eqref{eq:A-ns} extends smoothly to
$[0,\frac{n}{2}]$ with $A_{n,0}=1$.
Consequently,
\[
A_{n,s}=1+O_n(s),
\qquad
\kappa_{n,s}=\frac{2s}{c_n}A_{n,s}\le C_ns.
\]
Finally, since $(n-2s)p_s=2n$,   \eqref{eq:radial-prob} yields
\[
\begin{aligned}
\kappa_{n,s}
\left(
\int_{\rho>R}\rho^{-(n-2s)p_s}\,d\sigma
\right)^{\frac{1}{p_s}}&\le
\kappa_{n,s}
\left(c_n\int_R^2\rho^{-n-1}\,d\rho\right)^{\frac{1}{p_s}}\\
&\le
\kappa_{n,s}
\left(\frac{c_n}{n}\right)^{\frac{1}{p_s}}
R^{-\frac{n}{p_s}}
\le C_nsR^{-\frac{n}{p_s}}.
\end{aligned}
\]
Here the last constant is independent of $s$ and $R$, since
$0<\frac{1}{p_s}<\frac{1}{2}$.
This proves \eqref{eq:kernel-far}.
\end{proof}

The remaining result supplies the fixed-order regularity needed to evaluate the
integral equation at a maximum point.

\begin{lemma}
\label{lem:fixed-regularity}
Fix $0<s<\min\{1,\frac{n}{6}\}$.  Let $0\leq\gamma<1$,
$W\in L^\infty(\Sn)$, and   $u\in L^{p_s}(\mathbb{S}^n)$ solve, in the distributional sense,
\begin{equation}\label{eq:fixed-EL}
 (1-\gamma)P_su=\abs u^{p_s-2}u-\gamma W.
\end{equation}
Then
$ u\in L^\infty(\Sn)\cap C^{0,\beta}(\Sn)$
for some $\beta=\beta(n,s)>0$.  In particular, $u$ has a continuous
representative and $\abs u$ attains its maximum on $\Sn$.
\end{lemma}

\begin{proof}
Write $p=p_s$ and $a=\frac{n}{2s}$, and set
\[
V=(1-\gamma)^{-1}|u|^{p-2},
\qquad
g=-\gamma(1-\gamma)^{-1}W.
\]
Since $Vu+g\in L^{p'}(\mathbb S^n)$, the weak equation gives
\begin{equation}\label{eq:fixed-integral}
u=P_s^{-1}(Vu+g).
\end{equation}
Moreover, $a(p-2)=p$ and $u\in L^p(\mathbb S^n)$ imply $V\in L^a(\mathbb S^n)$.
\medskip

We claim that $u\in L^q(\mathbb S^n)$ implies $u\in L^Q(\mathbb S^n)$ whenever
\[
p\le q<Q<\infty,
\qquad
\frac1q-\frac1Q<\frac1a.
\]
Split $V=V_1+V_2$, where
$ V_1:=V\mathbf 1_{\{|V|>K\}}$ and
$V_2:=V\mathbf 1_{\{|V|\le K\}}$.
Then $V_2\in L^\infty(\mathbb S^n)$ and $\|V_1\|_a\to0$ as $K\to\infty$.
For a measurable function $W$, let $M_W$ denote the multiplication
operator
\[
 (M_W\varphi)(\omega):=W(\omega)\varphi(\omega).
\]
In particular,
$ P_s^{-1}M_{V_1}\varphi=P_s^{-1}(V_1\varphi)$.
For every finite $t\ge p$, H\"older inequality  gives
\[
\|P_s^{-1}(V_1\varphi)\|_t
\le C_{n,s,t}\|V_1\|_a\|\varphi\|_t.
\]
Choose $K$ so that $P_s^{-1}M_{V_1}$ has operator norm less
than $\frac{1}{2}$ on both $L^q(\mathbb S^n)$ and $L^Q(\mathbb S^n)$.
\medskip

Define $r$ by $\frac{1}{r}:=\frac{1}{Q}+\frac{1}{a}$.
Then $1<r<\min\{q,a\}$, so $V_2u+g\in L^r(\mathbb S^n)$ and HLS inequality yields
$P_s^{-1}(V_2u+g)\in L^Q(\mathbb S^n)\subset L^q(\mathbb S^n)$.
Let $\widetilde u\in L^Q(\mathbb S^n)$ be the unique solution of
\[(I-P_s^{-1}M_{V_1})\widetilde u=P_s^{-1}(V_2u+g)\]
 in $L^Q(\mathbb S^n)$.
Since $\widetilde u\in L^q(\mathbb S^n)$, both $u$ and $\widetilde u$
solve the same equation.
Uniqueness in $L^q(\mathbb S^n)$ therefore gives $u=\widetilde u$,
proving the claim.
\medskip

For any finite $R>p$, partition $[\frac{1}{R},\frac{1}{p}]$ into finitely
many intervals of length less than $\frac{1}{a}$.
Iterating the claim along the corresponding exponents gives
$u\in L^R(\mathbb S^n)$.
By finite-measure
 inclusion, we conclude that
$u\in L^q(\mathbb S^n)$
for every $1\le q<\infty$.
\medskip

It remains to obtain continuity.  Fix $r>a$.
Since $u\in L^{r(p-1)}(\mathbb S^n)$, we have $Vu+g\in L^r(\mathbb S^n)$.
The H\"older estimate for the order-$2s$ spherical potential
and \eqref{eq:fixed-integral} yield
\[
u\in C^{0,\beta}(\mathbb S^n),
\qquad
0<\beta<\min\left\{1,2s-\frac nr\right\}.
\]
Thus $u$ has a bounded continuous representative, and $|u|$
attains its maximum on $\mathbb S^n$.
\end{proof}

\section{Spectral comparison}
We record the spectral comparison used in the endpoint compactness argument.
\begin{lemma}\label{lem:spectral-domination}
For every $0<s<\frac{n}{2}$ and $\ell\geq0$,
\begin{equation*}
\frac{r_{s,\ell}-1}{2s}\geq\Lambda_\ell.
\end{equation*}
\end{lemma}
\begin{proof}
Set
\[
x=\frac n2+\ell,
\qquad
x_0=\frac n2,
\qquad
g(t)=\psi(x+t)-\psi(x_0+t).
\]Since $0<s<\frac{n}{2}$, all arguments below are positive for
$t\in[-s,s]$.
By \eqref{eq:frac-eigen} and the definition of $\psi$,
\[
\begin{aligned}
\log r_{s,\ell}
&=
\int_{-s}^{s}
\bigl[\psi(x+t)-\psi(x_0+t)\bigr]\,dt=
\int_{-s}^{s}g(t)\,dt.
\end{aligned}
\]
For $y>0$,
$\psi''(y)=-2\sum\limits_{j=0}^{\infty}(y+j)^{-3}$
is strictly increasing. Hence
\[
g''(t)=\psi''(x+t)-\psi''(x_0+t)
\ge0,
\]
so $g$ is convex on $[-s,s]$.
Jensen inequality for the
uniform probability measure on $[-s,s]$ implies
\[
\frac1{2s}\log r_{s,\ell}
=
\frac1{2s}\int_{-s}^{s}g(t)\,dt
\ge
g(0)
=
\psi\!\left(\ell+\frac n2\right)
-\psi\!\left(\frac n2\right)
=
\Lambda_\ell.
\]
Finally, applying $e^y-1\ge y$ with
\(y=\log r_{s,\ell}\), we obtain
\[
\frac{r_{s,\ell}-1}{2s}
\ge
\frac{\log r_{s,\ell}}{2s}
\ge
\Lambda_\ell.
\]
\end{proof}

{\bf Conflict of interest.} The authors declare that there is no conflict of interest.
\medskip

{\bf Data availability.} No data were used in the research described in this article.


\begin{thebibliography}{99}

\bibitem{BakryGentilLedoux2014} D. Bakry, I. Gentil and M. Ledoux, \textit{Analysis and Geometry of Markov Diffusion Operators}, Grundlehren der mathematischen Wissenschaften 348, Springer, Cham, 2014.

\bibitem{BartschWethWillem2003}
T. Bartsch, T. Weth and M. Willem,
\textit{A Sobolev inequality with remainder term and critical equations on domains with topology for the polyharmonic operator},
Calc. Var. Partial Differential Equations \textbf{18}
(2003), 253-268.

\bibitem{Beckner1992}
W. Beckner,
\textit{Sobolev inequalities, the Poisson semigroup, and analysis on the sphere
$\mathbb S^n$},
Proc. Natl. Acad. Sci. USA \textbf{89} (1992), 4816-4819.

\bibitem{Beckner1993}
W. Beckner,
\textit{Sharp Sobolev inequalities on the sphere and the Moser-Trudinger
inequality},
Ann. of Math. \textbf{138} (1993), 213-242.

\bibitem{Beckner1997}
W. Beckner,
\textit{Logarithmic Sobolev inequalities and the existence of singular
integrals},
Forum Math. \textbf{9} (1997), 303-324.

\bibitem{BianchiEgnell1991}
G. Bianchi and H. Egnell,
\textit{A note on the Sobolev inequality},
J. Funct. Anal. \textbf{100} (1991), 18-24.

\bibitem{BrezisLieb1985}
H. Brezis and E. Lieb,
\textit{Sobolev inequalities with remainder terms},
J. Funct. Anal. \textbf{62} (1985), 73-86.



\bibitem{Carlen2017}
E. A. Carlen, \textit{Duality and stability for functional inequalities},
Ann. Fac. Sci. Toulouse Math. \textbf{26} (2017), 319-350.

\bibitem{ChenFrankWeth2013}
S. Chen, R. L. Frank and T. Weth,
\textit{Remainder terms in the fractional Sobolev inequality},
Indiana Univ. Math. J. \textbf{62} (2013), 1381-1397.

\bibitem{ChenLuTang2023}
L. Chen, G. Lu and H. Tang,
\textit{Sharp stability of log-Sobolev and Moser-Onofri inequalities on the sphere},
J. Funct. Anal. \textbf{285} (2023), 110022.

\bibitem{ChenLuTang2024}
L. Chen, G. Lu and H. Tang,
\textit{Stability of Hardy-Littlewood-Sobolev inequalities with explicit lower
bounds},
Adv. Math. \textbf{450} (2024), 109778.



\bibitem{ChenLuTang2025}
L. Chen, G. Lu and H. Tang,
\textit{Optimal asymptotic lower bound for stability of fractional Sobolev
inequality and the global stability of log-Sobolev inequality on the
sphere},
Adv. Math. \textbf{479} (2025), 110438.

\bibitem{ChenLuTang2026MA}
L. Chen, G. Lu and H. Tang,
\textit{Optimal stability of Hardy-Littlewood-Sobolev and Sobolev inequalities  of arbitrary orders with dimension-dependent constants},
Math. Ann. \textbf{394} (2026), 77.

\bibitem{ChenLuTang2026CP}
L. Chen, G. Lu and H. Tang,
\textit{Stability for critical points of the Hardy-Littlewood-Sobolev
inequality and a dual stability framework},
arXiv:2605.19288 (2026).

\bibitem{Clarke1983}
F. H. Clarke,
\textit{Optimization and Nonsmooth Analysis},
Wiley, New York, 1983.

\bibitem{Danskin1967}
J. M. Danskin,
\textit{The Theory of Max-Min and Its Application to Weapons
Allocation Problems},
Springer-Verlag, Berlin, 1967.


\bibitem{DolbeaultEtAl2025}
J. Dolbeault, M. J. Esteban, A. Figalli, R. L. Frank and M. Loss,
\textit{Sharp stability for Sobolev and log-Sobolev inequalities,
with optimal dimensional dependence},
Camb. J. Math. \textbf{13} (2025), 359-430.



\bibitem{FrankKonigTang2020}
R. L. Frank, T. K\"onig and H. Tang,
\textit{Classification of solutions of an equation related to a conformal log
Sobolev inequality},
Adv. Math. \textbf{375} (2020), 107395.

\bibitem{Gross1975}
L. Gross,
\textit{Logarithmic Sobolev inequalities},
Amer. J. Math. \textbf{97} (1975), 1061-1083.



\bibitem{Hersch1970}
J. Hersch,
\textit{Quatre propri\'et\'es isop\'erim\'etriques de membranes sph\'eriques
homog\`enes},
C. R. Acad. Sci. Paris S\'er. A \textbf{270} (1970), A1645-A1648.


\bibitem{JinLiXiong2014}
T. Jin, Y. Y. Li and J. Xiong,
\textit{On a fractional Nirenberg problem, Part I: blow up analysis and compactness of solutions},
J. Eur. Math. Soc. \textbf{16} (2014), 1111-1171.

\bibitem{Konig2023}
T. K\"onig,
\textit{On the sharp constant in the Bianchi-Egnell stability inequality},
Bull. Lond. Math. Soc. \textbf{55} (2023), 2070-2075.

\bibitem{Konig2025}
T. K\"onig,
\textit{Stability for the Sobolev inequality: existence of a minimizer},
J. Eur. Math. Soc. (online first, 2025),
\href{https://doi.org/10.4171/JEMS/1582}{doi:10.4171/JEMS/1582}.

\bibitem{Laugesen2021}
R. S. Laugesen,
\textit{Well-posedness of Hersch-Szeg\H{o}'s center of mass
by hyperbolic energy minimization},
Ann. Math. Qu\'ebec \textbf{45}  (2021), 363-390.
\bibitem{Lieb} E. H. Lieb, {\it Sharp constants in the Hardy-Littlewood-Sobolev and related inequalities}, Ann. Math., \textbf{118}
(1983), 349-374.
\bibitem{LuWei1999}
G. Lu and J. Wei,
\textit{On a Sobolev inequality with remainder terms},
Proc. Amer. Math. Soc. \textbf{128} (2000), 75-84.

\end{thebibliography}
\end{document}